\documentclass[amssymb,amscd,amsfonts,refcheck,11pt,graphicx,verbatim,righttag]{amsart}
 \usepackage{mathrsfs}

\usepackage{amsfonts}
\usepackage{amssymb}
\usepackage{verbatim}

 \numberwithin{equation}{section}
 \allowdisplaybreaks[4]
\theoremstyle{plain}
 \def\eps{\varepsilon}

\newtheorem{theorem}{Theorem}[section]
\newtheorem{corollary}{Corollary}

\newtheorem{lemma}[theorem]{Lemma}
\newtheorem{proposition}{Proposition}

\theoremstyle{definition}

\newtheorem{remark}{Remark}

\newcommand{\N}{{\mathbb N}}
\newcommand{\Z}{{\mathbb Z}}
\newcommand{\R}{{\mathbb R}}

\newcommand{\Q}{{\mathbb Q}}
\newcommand{\C}{{\mathcal C}}
\newcommand{\D}{{\mathcal D}}

\newcommand{\M}{{\mathcal M}}

\title[Localized asymptotic
      behavior for potentials]
      {Localized asymptotic
      behavior for almost additive potentials}

\author{Julien Barral}
\address{LAGA (UMR 7539), D\'epartement de Math\'ematiques, Institut Galil\'ee, Universit\'e
 Paris 13, 99 avenue Jean-Baptiste Cl\'ement , 93430  Villetaneuse, France}
\email{barral@math.univ-paris13.fr}
\author{Yan-Hui Qu}
\address{Department of Mathematics, Tsinghua University, Beijing 100084, China
}
\email{jyh02@mails.tsinghua.edu.cn}
\subjclass{Primary: 37B40; Secondary: 28A80.}
 \keywords{Almost additive potential, weak concavity, multifractal analysis.}

\begin{document}
\maketitle


\begin{abstract}
We conduct the multifractal analysis of the level sets of the
asymptotic behavior of almost additive continuous potentials
$(\phi_n)_{n=1}^\infty$ on a topologically mixing subshift of finite
type $X$ endowed itself with a metric associated with such a
potential. We work without additional regularity  assumption other
than continuity. Our approach differs from those used previously to
deal with this question under stronger assumptions on the
potentials. As a consequence, it provides
 a new description of the structure of the spectrum in terms of
 {\it weak} concavity. Also, the lower bound for the spectrum is
  obtained as a consequence of the study sets of points at which
   the asymptotic behavior of $\phi_n(x)$ is localized, i.e.
  depends on the point $x$ rather than being equal to a constant.
  Specifically, we compute the Hausdorff dimension of sets of the form
   $\{x\in X: \lim_{n\to\infty} \phi_n(x)/n=\xi(x)\}$, where $\xi$
    is a given continuous function.
   This has interesting geometric applications to fixed points
   in the asymptotic average for dynamical systems in $\R^d$,
    as well as the fine local behavior of the harmonic measure
     on conformal planar Cantor sets.
\end{abstract}

\section{Introduction} We say that $(X,T)$ is a {\it topological dynamical system}
(TDS) if $X$ is a compact metric space and $T$ is a continuous
mapping from $X$ to itself. We denote by $\M(X,T)$ the set of
invariant probability measures on~$(X,T)$.

We say that $\Phi=(\phi_n)_{n=1}^\infty$ is {\it almost additive} if
$\phi_n$ is continuous from $X$ to $\R$  and there is a positive
constant $C(\Phi)>0$ such that
\begin{equation}\label{aa}
-C(\Phi)+ \phi_n+\phi_p\circ T^n\leq \phi_{n+p}\leq C(\Phi)+
\phi_n+\phi_p\circ T^n,\quad \forall\,n,p\in \N.
\end{equation}
Typical examples are the additive potential given by the sequence of
Birkhoff sums
 $(S_n\varphi=\sum_{k=0}^{n-1}\varphi\circ T^k)_{n\ge 1}$ of
 a continuous function $\varphi: X\to \R$, and more generally sequences
 of the form  $(\log \|S_nM\|)_{n\ge 1}$, where  $(S_nM)_{n\ge 1}$ is
   the  sequence of Birkhoff products
   $(M\circ T^{n-1})\cdots (M\circ T)\cdot M$ associated
   with a continuous function $M$ from $X$ to the set of positive square matrices.

By subadditivity, for every $\mu\in \M(X,T)$, $\displaystyle
\Phi_*(\mu):=\lim_{n\to\infty} \int_X \frac{\phi_n}{n}\,
\text{d}\mu$ exists, and we define the compact  convex set $
L_\Phi=\{ \Phi_*(\mu):\mu\in \M(X,T)\}. $ We denote by
$\C_{aa}(X,T)$ the collection of almost additive potentials on~$X$.

The ergodic theorem naturally raises the following question. Given
$\Phi$ an almost additive potential taking values in $\R^d$ (this
means that $\Phi=(\Phi^1,\cdots,\Phi^d)$ with each $\Phi^i\in
\C_{aa}(X,T)$) and $\xi:X\to\R^d$ a continuous function, what is the
Hausdorff dimension of the set
$$
E_\Phi(\xi):=\Big \{x\in X:
\lim_{n\to\infty}\frac{\phi_n(x)}{n}=\xi(x)\Big \}?
$$
When $\xi(x)\equiv\alpha$ is constant, this question has been solved
for some $C^{1+\varepsilon}$ conformal dynamical systems, sometimes
assuming restrictions on the regularity of $\Phi$, and this problem
is known as the multifractal analysis of Birkhoff averages, and more
generally almost additive potentials \cite{Collet,Rand,PeW,PW,P,
O,FF,BS,FFW,BSS,FLW,FeLa02,Fen03,Olsen,BD}. Moreover,
 the optimal results are expressed in terms of a variational
 principle of the following form:
 $E_\Phi(\alpha)\neq \emptyset$ if and only
 if $\alpha\in L_\Phi=\{ \Phi_*(\mu)=
 (\Phi^1_*(\mu),\dots,\Phi^d_*(\mu)):\mu\in \M(X,T)\}$ and in this case
\begin{equation}\label{varprinc}
\dim_H E_\Phi(\alpha)=\max\left \{\frac{h_\mu(T)}{\int_X \log \|DT\|
\, \text{d}\mu}: \mu\in \M(X,T),\  \Phi_*(\mu)=\alpha\right \},
\end{equation}
where $h_\mu(T)$ is the measure theoretic entropy of $\mu$ relative
to $T$ (see \cite{BF09} for such a result in a non-conformal context).

To our best knowledge no result is known for $\dim_H E_\Phi(\xi)$
for non constant $\xi$. We are going to give an answer to this
question when $(X,T)$ is a topologically mixing subshift of finite
type endowed with a metric associated with a negative almost
additive potential, and then transfer our result to geometric
realizations on Moran sets like those studied in \cite{Bar96}, the
main examples being $C^1$ conformal repellers and $C^1$ conformal
iterated function systems (see section~\ref{examples} for precise
definitions and statements). In the setting outlined above, if $d=1$
and $\xi$ takes its values in $L_\Phi$, we find the natural
variational formula
$$
\dim_H  E_\Phi(\xi)=\max\left \{\frac{h_\mu(T)}{\int_X \log \|DT\|
\, \text{d}\mu}: \mu\in \M(X,T),\  \Phi_*(\mu)\in \xi(X)\right\}.
$$
As application of this kind of results, we obtain unexpected
properties like the following one:
 Let $d\in\N_+$ and $(m_1,\dots,m_d)$ be $d$ integers $\ge 2$.
 Let $T:[0,1]^d\to[0,1]^d$ be the mapping
  $(x_1,\dots,x_d)\mapsto (m_1x_1\pmod 1,\dots,m_d x_d\pmod 1)$. Consider
$$
\mathcal F:=\Big\{x\in [0,1]^d: \lim_{n\to\infty}
\frac{1}{n}\sum_{k=0}^{n-1} T^k x=x\Big \},
$$
the set of those points $x$ which are fixed by $T$ in the asymptotic
average. Then $\mathcal F$ is dense and of full Hausdorff dimension
in $[0,1]^d$.

Another application concerns harmonic measure. Let us consider here
the special case of the set $J=C\times C\subset \R^2$, where $C$ is
the middle third Cantor set. The harmonic measure on $J$ is the
probability measure $\omega$ such that for each
 $x\in J$ and $r>0$, $\omega (B(x,r))$ is the probability that a planar Brownian
  motion started at $\infty$ attains $J$ for the first time at a point of $B(x,r)$
  (see Section \ref{appli2} for more general examples and a reference). For $x\in J$,
   one defines the local dimension of $\omega$ at $x$ as
$\displaystyle d_\omega (x)=\lim_{r\to 0^+} {\log \omega
(B(x,r))}/{\log r} $ whenever this limit exists. Let $I$ stand for
the set of all possible local dimensions
 for $\omega$. By using the fact that $\omega$ is a Gibbs measure,  we prove that
  if $\xi:J\to\R_+$ is continuous and $\xi(J)\subset I$, then the set
  $E_\omega(\xi)=\{x\in J: d_\omega(x)=\xi(x)\}$ is dense
  in $J$ and the following variational formula holds:
$$
\dim_HE_\omega(\xi)=\sup\{\dim_H E_\omega(\alpha):\alpha\in\xi(J)\},
\text{ where $ E_\omega(\alpha)=\{x\in J: d_\omega(x)=\alpha\}.$}
$$

Our approach necessitates to revisit the case where $\xi$ is
constant.
 This brings out an interesting new property
 of the Hausdorff spectrum
 $\alpha\mapsto \dim_H E_\Phi(\alpha)$. We
 call this property {\it weak} concavity;
 it is between concavity and quasi-concavity. This structure turns out to be crucial
 in establishing our results on fixed points in the asymptotic average.

The paper is organized as follows. In Section~\ref{statements} we
give basic definitions and recalls about thermodynamic formalism,
and then state our main results  on subshift of finite type.  In
Section~\ref{examples} we give the  geometric realizations.~The
other sections are dedicated~to~the proofs.

\section{Definitions,
and main results on subshifts of finite type}\label{statements}

Section~\ref{firstdefi} introduces some additional notions related
to almost additive potentials, Section~\ref{defmetric} introduces
the metrics we will put on topologically mixing subshifts of finite
types, while Section~\ref{TF} recalls the variational principle for
almost additive potentials. Then Section~\ref{defspectra} introduces
two fundamental dimension functions in the multifractal analysis of
almost additive potentials, as well as a notion of weak concavity.
Finally Section~\ref{resultats} provides our main results on
topologically mixing subshifts of finite types.

\subsection{Definitions}

\subsubsection{Vector-valued almost additive
 potentials and some associated quantities}\label{firstdefi}

Given $\Phi\in \C_{aa}(X,T),$ define
$\Phi_{\max}:=\max(\phi_{1})+C(\Phi) $ and
$\Phi_{\min}:=\min(\phi_{1})-C(\Phi)$.

Define two  collections of special almost additive potentials on $X$
as $$
 \C_{aa}^+(X,T):=\{\Phi\in \C_{aa}(X,T): \Phi_{\min}>0\}$$ and
 $$ \C_{aa}^-(X,T):=\{\Phi\in \C_{aa}(X,T): \Phi_{\max}<0\}.
$$ These sets contain in particular the sequences of Birkhoff sums of
positive continuous functions and negative continuous functions
respectively.

For $\Phi\in\C_{aa}^-(X,T)$ we get $ \phi_{n+1}(x)\leq
\phi_n(x)+\phi_1(T^nx)+C(\Phi)\leq \phi_n(x)+\Phi_{\max}<\phi_n(x),
$ So $\{\phi_n:n\in\N\}$ is a strictly decreasing sequence of
functions.

If $\Phi=(\Phi^1,\cdots,\Phi^d)$ is such that each
$\Phi^j\in\C_{aa}(X,T)$, then we call $\Phi$ a {\it vector-valued
almost additive potential} and write $\Phi\in\C_{aa}(X,T,d)$. We
have $\Phi=(\phi_n)_{n=1}^\infty$ with
$\phi_n=(\phi_n^1,\cdots,\phi_n^d).$ If $\phi:\Sigma_A\to \R^d$ is
continuous, we define
\begin{equation}\label{normn}
\|\phi\|_n:=\sup_{x|_n=y|_n} |\phi(x)-\phi(y)|,
\end{equation}
where $|u|$ stands for the euclidean norm of $u$. For
$\Phi\in\C_{aa}(X,T,d)$,  let $\|\Phi\|_n:=\|\phi_n\|_n$.

\subsubsection{Weak Gibbs metric on subshifts of finite type}\label{defmetric}
Let $(\Sigma_A,T) $ be a topologically mixing subshift of finite
type over the alphabet $\{1,\cdots,m\}$, where $A$ is a $m\times m$
matrix with entries $0$ and $1$ such that $A^{p_0}>0$ for some
$p_0\in\N$ and $T$ is the shift map. We shall endow $\Sigma_A$ with
a metric $d_\Psi$ naturally associated with a potential
 $\Psi\in\mathcal{C}_{aa}^-(\Sigma_A,T)$. This kind of
 metrics have been considered in
 \cite{GP97} and \cite{KS04} associated with negative additive
 potentials in order to transfer to the symbolic side
 the study of some $C^1$ hyperbolic dynamics.

 Let $\Sigma_{A,n}$ be the set of the admissible words of length $n$
 and let
$\Sigma_{A,\ast}:=\bigcup_{n\geq 0} \Sigma_{A,n}.$ For
$w\in\Sigma_{A,\ast} $ and $w=w_1\cdots w_n$, we denote the length
of $w$ by $|w|=n$. Given $w\in \Sigma_{A,\ast}\cup \Sigma_{A}$ with
$|w|\geq n$, we denote $w_1\cdots w_n$ by $w|_n$.  Given
$u\in\Sigma_{A,\ast} $ and $v\in \Sigma_{A,\ast} \cup\Sigma_{A}$, if
$u_j=v_j$ for $j=1,\cdots,|u|$, then we say $u$ is a {\it prefix }
of $v$ and write $u\prec v.$ For $u=u_1\cdots u_n\in \Sigma_{A,n}$,
$u^\ast$ stands for $u|_{n-1}$. For $x, y\in \Sigma_{A,\ast}\cup
\Sigma_A$ such that $x\ne y$, $x\wedge y$ stands for the common
prefix of $x$ and $y$ of maximal length. Given $w\in\Sigma_{A,n}$,
the cylinder $[w]$ is defined as
$$
[w]:=\{x\in\Sigma_A: x|_n=w\}.
$$

Recall that  $A^{p_0}(i,j)>0$ for all $1\leq i,j\leq m$,
consequently $A^{p_0+2}(i,j)>0$. For each $i,j$ we fix
 $w(i,j)\in \Sigma_{A,p_0}$ such that $iw(i,j)j$ is admissible.
Define
\begin{equation}\label{big-xi}
\mathcal{W}:=\{w(i,j): 1\leq i,j\leq m\}.
\end{equation}

For $\Phi\in \C_{aa}(\Sigma_A,T)$ and $w\in\Sigma_{A,n}$ we define
$$
\Phi[w]:=\sup\{\exp({\phi_n(x)}):x\in [w]\}.
$$

Now we fix a $\Psi\in  \C_{aa}^-(\Sigma_A,T).$ For $x,y\in \Sigma_A$
define
$$
d_\Psi(x,y):=
\begin{cases} \Psi[x\wedge y], & \text{ if }
x\ne y\\
0, & \text{ if } x=y.\end{cases}
$$
\begin{proposition}\label{metric-sym}
 $d_\Psi$ is an ultra-metric  distance on $\Sigma_A.$
 If $x\in \Sigma_A$ and $r>0$, the closed ball
$\overline{B(x,r)}$ is the cylinder $[x|_n]$, where $n$ is the
unique integer such that $\Psi[x|_{n-1}]> r$ and $\Psi[x|_n]\leq r$.
Each cylinder $[w]$ is a ball with $\mathrm{diam}([w])=\Psi[w]$.
\end{proposition}

The proof is elementary and we omit it.

For the metric space $(\Sigma_A,d_\Psi)$ we define
$${\mathcal B}_n(\Psi)=\{w\in\Sigma_{A,\ast}:[w]
\text{ is a closed ball of } \Sigma_A \text{ with radius }
e^{-n}\}\quad(n\ge 0).$$ It is clear that $\{[w]:w\in{\mathcal
B}_n(\Psi)\}$ is a covering of $\Sigma_A$ for each $n\ge 0$.

If we take $\Psi=(-n\log m)_{n\geq 1}$, it is ready to check that
 $d_\Psi(x,y)=m^{-|x\wedge y|},$ which is the standard metric on
 $\Sigma_A$. We denote this special metric by $d_1.$

 \subsubsection{Recalls on the thermodynamic formalism}\label{TF}
 The thermodynamic formalism for almost additive potentials has been
studied in several works
~\cite{Fal88,Bar96,FeLa02,Fen04,B,Mum06,BD,CFH}. For our purpose, we
only need to consider
 the subshift of finite type case. Let $(\Sigma_A,T)$ be a topologically mixing
  subshift of
finite type.
 Given
$\Phi\in\C_{aa}(\Sigma_A,T)$, the topological pressure can be
defined as
\begin{equation}\label{pressure-topo}
P(T,\Phi):=\lim_{n\to\infty}\frac{1}{n}\log \sum_{w\in \Sigma_{A,n}}
\exp(\sup_{x\in[w]}\phi_n(x)).
\end{equation}
The following extension of the classical variational principle valid
for additive continuous potentials (see \cite{Ruelle}) holds:

\begin{theorem}{\cite{B,BD,CFH}}\label{varia-principle}
Let $(\Sigma_A,T)$ be a topologically mixing subshift of finite
type. For any $\Phi\in\C_{aa}(\Sigma_A,T)$, we have $
P(T,\Phi)=\sup\{h_\mu(T)+\Phi_\ast(\mu): \mu\in\M(\Sigma_A,T)\}. $
\end{theorem}

If $\mu\in \M(\Sigma_A,T)$ such that
$P(T,\Phi)=h_\mu(T)+\Phi_\ast(\mu),$ then $\mu$ is called an
equilibrium state of $\Phi.$ It is shown in \cite{B} that every
$\Phi\in \C_{aa}(\Sigma_A,T)$ has an equilibrium state (in fact the
result holds for more general TDS). Define
 \begin{equation}\label{equilibrium}
 U(\Sigma_A,T):=\{\Phi\in
\C_{aa}(\Sigma_A,T): \Phi \text{ has a unique equilibrium state}\}.
 \end{equation}
For instance, this set contains the sequence of Birkhoff sums of any
H\"older continuous function when $\Sigma_A$ is endowed with a
metric $d_\Psi$ (see \cite{Bow75}).

 For $\Phi^1,\cdots,\Phi^k\in \C_{aa}(\Sigma_A,T)$ and
 $q=(q_1,\cdots,q_k)\in\R^k$, define
 $$
 F(q):=P(T,q_1\Phi^1+\cdots+q_k\Phi^k).
 $$
 It is shown in \cite{BD} that if ${\rm span}\{\Phi^1,\cdots,\Phi^k\}\subset
  U(\Sigma_A,T)$, then $F(q)$ is convex and in $C^1(\R^k).$

\subsection{Two dimension functions; weak concavity}\label{defspectra}
Let us recall what is the range of those $\alpha$ such that
$E_\Phi(\alpha)\neq\emptyset$.

\begin{proposition}[\cite{FH}]\label{bsic-aa}
Let $\Phi\in\C_{aa}(\Sigma_A,T,d)$. We have $E_\Phi(\alpha)\ne
\emptyset$ if and only if $\alpha\in L_\Phi.$
\end{proposition}

Now we introduce two functions which will turn out to take the same
values on $L_\Phi$ and provide the
 Hausdorff and packing dimensions of
the sets $E_\Phi(\alpha)$. They correspond to different point of
views to estimate these dimensions, namely box-counting of balls
intersecting $E_\Phi(\alpha)$ and variational principle for entropy
like \eqref{varprinc}. The proofs of the propositions stated in this
section are given in Section~\ref{proofs}.

\noindent {\bf (1) Box-counting type function; weakly concave  large
deviation spectrum:} fix $\Psi\in \C_{aa}^-(\Sigma_A,T)$ and
$\Phi\in \C_{aa}(\Sigma_A,T,d)$. Define  $d_\Psi$ and ${\mathcal
B}_n(\Psi)$ as above. Given  $\alpha\in L_\Phi$, $n\geq 1$ and
$\epsilon>0$, define
$$F(\alpha,n,\epsilon,\Phi,\Psi):=\Big \{u\in {\mathcal B}_n(\Psi):\
\text{ there exists}\  x\in[u] \ \text{such that }\  \Big
|\frac{\phi_{|u|}(x)}{|u|}-\alpha\Big |<\epsilon \Big \}.
$$

Let $f(\alpha,n,\epsilon,\Phi,\Psi)$ be the cardinality of
$F(\alpha,n,\epsilon,\Phi,\Psi)$.

\begin{proposition}\label{dim-formular-1} For any
$\Psi\in \C_{aa}^-(\Sigma_A,T)$, the limit
\begin{equation}\label{full-dim}
D(\Psi)=\lim_{n\to\infty}\frac{\log \#{\mathcal B}_n(\Psi)}{n}
\end{equation}
exists.  Moreover
\begin{equation}\label{ful-dim-1}
 D(\Psi)\leq (1+1/|\Psi_{\max}|)\log m.
\end{equation}
 For any $\Phi\in
\C_{aa}(\Sigma_A,T,d)$ and any $\alpha\in L_\Phi$, we have
\begin{equation}\label{large-div}
\lim_{\epsilon\to 0}\liminf_{n\to\infty} \frac{\log
f(\alpha,n,\epsilon,\Phi,\Psi)}{n}= \lim_{\epsilon\to
0}\limsup_{n\to\infty}\frac{\log
f(\alpha,n,\epsilon,\Phi,\Psi)}{n}=:\Lambda_{\Phi}^{\Psi}(\alpha)=\Lambda(\alpha).
\end{equation}
The function $\Lambda:L_\Phi\to \R$ is  upper semi-continuous.
\end{proposition}

We will prove that $\Lambda(\alpha)$ is the Hausdorff dimension of
$E_\Phi(\alpha)$ for all $\alpha\in L_\Phi$. The function $\Lambda$
has more regularity than upper semi-continuity. To make this precise
we  need several standard notations from convex analysis. Given
$A\subset\R^d$, the {\it affine hull} of $A$ is the smallest affine
subspace of $\R^d$ containing $A$ and is denoted by
$\mathrm{aff}(A).$ For a convex set $A$, we define $\mathrm{ri}(A),$
the {\it relative interior} of $A$ as
$$
\mathrm{ri}(A):=\{x\in \mathrm{aff}(A): \ \exists \epsilon>0,
(x+\epsilon B)\cap \mathrm{aff}(A)\subset A\},
$$
 where $B=B(0,1)\subset \R^d$ is the
unit open ball.  Let $A\subset \R^d$ be a convex set and $h: A\to
\R$ be a function. If
 there exists $c\geq 1$ such that for any $\alpha,\beta\in A$, we can
find $\gamma_1=\gamma_1(\alpha,\beta),
\gamma_2=\gamma_2(\alpha,{\beta})\in [c^{-1},c]$ such that for any
$\lambda\in[0,1]$
\begin{equation}\label{lower-semi-conti}
\lambda h(\alpha)+(1-\lambda)h(\beta)\leq h\Big
(\frac{\lambda\gamma_1\alpha+(1-\lambda)\gamma_2\beta}
{\lambda\gamma_1+(1-\lambda)\gamma_2}\Big ),
\end{equation}
then we call $h$  a {\it weakly concave function} on $A.$ Note that
if $c=1$, we go back to the usual concept of concave function. Also,
$h(\gamma)\ge \min (h(\alpha),h(\beta))$ if $\gamma\in
[\alpha,\beta]\subset A$, thus $h$ is quasi-concave.

 \begin{proposition}\label{regularity}
  The function $\Lambda: L_\Phi\to \R$ is
 bounded, positive and weakly concave. It is
 continuous on any closed interval $I\subset
 L_\Phi$ and  on $\mathrm{ri}(A)$, where $A\subset L_\Phi$ is any convex set.
  Consequently it is continuous on
  ${\rm ri}(L_\Phi).$ If moreover $L_\Phi$ is a convex polyhedron,
 then $\Lambda$ is continuous on $L_\Phi.$
  Assume $I=[\alpha_0,\alpha_1]\subset L_\Phi$  and $\alpha_{\max}\in
 I$ such that $\Lambda(\alpha_{\max})=
 \max\{\Lambda(\alpha): \alpha\in
 I\}$, then $\Lambda$ is decreasing from $\alpha_{\max}$ to
 $\alpha_j$, $j=0,1.$
 \end{proposition}

 \begin{remark}{\rm  Large deviations spectra for the
 Hausdorff dimension estimation of sets like $ E_\Phi(\alpha)$
 have been considered since the first studies of
 multifractal properties of Gibbs or weak Gibbs measures
  and then extended to the study of Birkhoff averages
  \cite{Collet,Rand,BMP,PeW,PW,P,BS,O,FF,FFW,BSS,FLW}.
  Until now, in the situations where such a spectrum may be
  non-concave \cite{BSS,BD,FLW},
  no description of its regularity like that of
  Proposition~\ref{regularity} had been given.
   Moreover, the methods used in the papers mentioned above seem
   not adapted to provide this information.
 }\end{remark}

\noindent {\bf (2) Function associated with a conditional
variational principle:}  For $\alpha\in L_\Phi$ let
$$
{\mathcal E}(\alpha)={\mathcal E}_\Phi^\Psi(\alpha):=\sup\left
\{\frac{h_\mu(T)}{-\Psi_\ast(\mu)}:\mu\in\M(\Sigma_A,T)\ \
\text{such that } \ \ \Phi_\ast(\mu)=\alpha\right\}.
$$

\begin{remark}{\rm When $\Phi$ and $\Psi$ are clear from the context, most
of the time we simplify the notations $\Lambda_\Phi^\Psi$,
${\mathcal E}_\Phi^\Psi$, $F(\alpha,n,\epsilon,\Phi,\Psi),$
$f(\alpha,n,\epsilon,\Phi,\Psi)$  to $\Lambda$, ${\mathcal E},$
$F(\alpha,n,\epsilon), f(\alpha,n,\epsilon)$ respectively. We use
the full notations only when we want to emphasize the $\Phi$- and
$\Psi$-dependence of the quantities.

}
\end{remark}

\subsection{Main results on topologically mixing subshift of finite
type}\label{resultats}\

Throughout this subsection we fix $\Phi\in \C_{aa}(\Sigma_A,T,d)$
and $\Psi\in \C_{aa}^-(\Sigma_A,T)$. We work on the metric space
$(\Sigma_A,d_\Psi)$. If $E\subset (\Sigma_A,d_\Psi)$, $\dim_H E,
\dim_P E, \dim_B E $ stand for its Hausdorff, packing and box
dimensions respectively. To not
 assuming additional regularity  assumption
  for $\Phi$ and $\Psi$  is natural, since this
  flexibility on $\Psi$ makes it possible to describe a larger
  class of geometric realizations of the next results,
  and there is no special reason to considers the sets
   $E_\Phi(\xi)$ under restrictions on $\Phi$.
   However,  the proofs will use extensively approximations of almost
   additive potentials by H\"older potentials.

For convenience we write $\D(\alpha)=\D_\Phi(\alpha):=\dim_H
E_\Phi(\alpha)$.


 \begin{theorem}[{\bf Multifractal analysis of the level sets
 $E_\Phi(\alpha)$}]\label{main-one-sided}$\ $

\begin{enumerate}
\item $E_\Phi(\alpha)\ne\emptyset$ if and only if $\alpha\in L_\Phi$.
For  $\alpha\in L_\Phi$ we have
$$
\D(\alpha)= \Lambda(\alpha)={\mathcal E}(\alpha),$$ and the function
$\D$ is weakly concave.

\item $\dim_H\Sigma_A=\dim_B\Sigma_A=D(\Psi)=
\max\{\D(\alpha): \alpha\in L_\Phi\}$.
\end{enumerate}
\end{theorem}

 \begin{theorem}[{\bf Localized asymptotic behavior}]\label{main-fun-level-one-sided}
 Assume $\xi:\Sigma_A\to\R^d$ is continuous
 and
$\xi(\Sigma_A)\subset \mathrm{aff}(L_\Phi)$.

\begin{enumerate}
\item $\dim_H E_\Phi(\xi)\geq
\sup\{\D(\alpha):\alpha\in \xi(\Sigma_A)\cap \mathrm{ri}(L_\Phi)\}.$

\item If $\xi(\Sigma_A)\subset L_\Phi$
then $E_\Phi(\xi)$ is dense in $\Sigma_A$.

\item If
$\sup\{\D(\alpha):\alpha\in\xi(\Sigma_A)\cap \mathrm{ri}(L_\Phi)\}=
\sup\{\D(\alpha):\alpha\in \xi(\Sigma_A)\cap L_\Phi\}$, then
$$
{\dim}_H E_\Phi(\xi)=\dim_P E_\Phi(\xi)=\sup\{\D(\alpha):\alpha\in
\xi(\Sigma_A)\cap L_\Phi\}.
$$

\item If $d=1$ and $\xi(\Sigma_A)\subset L_\Phi$, then $E_\Phi(\xi)$
is dense and
$$
 \dim_H E_\Phi(\xi)=\dim_P
E_\Phi(\xi)=\sup\{\D(\alpha):\alpha\in \xi(\Sigma_A)\}.
$$
\end{enumerate}
\end{theorem}

\begin{remark}\label{connexion}{\rm

(1) In ~\cite{BSS,BD}, assuming that
$$
{\rm span}\{\Phi^1,\cdots,\Phi^d,\Psi\}\subset
  U(\Sigma_A,T),
$$ the equality $\D(\alpha)= {\mathcal E}(\alpha)$ is shown only for
$\alpha\in \mathrm{int}(L_\Phi)$, where ${\rm int}(L_\Phi)$ denotes
the interior of $L_\Phi$. The argument is strongly based on the
differentiability of the related pressure functions in these cases.

\noindent (2) In \cite{FLW}, the authors consider the case of
additive potentials $\Phi$ and $\Psi$, and work under the assumption
that $\Psi$ corresponds to a H\"older potential. They show
$\D(\alpha)= {\mathcal E}(\alpha)$ for all $\alpha\in L_\Phi$. Here
we work under weaker  assumptions, i.e.  both $\Phi$ and $\Psi$ are
almost additive.
 Also, we use a different method
 to compute the function $\D(\alpha)$,
 namely concatenation of Gibbs measures.
 Such a method has been used successfully in \cite{KS04}
 to deal with the special sets
 $E_\Psi(\alpha)$ when $\Psi$ is additive ( i.e. taking $\Phi=\Psi;$ notice that
 in this case the spectrum is
 always concave). Here,
 we need to refine such an approach in order to remove some delicate
 points in our geometric
  application to attractors of $C^1$ conformal iterated function systems.

 }\end{remark}

\begin{remark}\label{remloc}
{\rm \noindent (1) The proof of
Theorem~\ref{main-fun-level-one-sided} uses the weak concavity of
the spectrum $\D$. It also requires to concatenate
 Gibbs measures in a more elaborated way than to determine~$\D$.

 \noindent(2) In fact we shall prove a slightly more general result
 than Theorem~\ref{main-fun-level-one-sided}(1):

 (1') Suppose that
 $\xi$ is  continuous
outside a subset $E$ of $\Sigma_A$, bounded  and
$\xi(\Sigma_A)\subset \mathrm{aff}(L_\Phi)$.
 If $\dim_H E<\sup\{\D(\alpha):\alpha\in \xi
 (\Sigma_A\setminus E)\cap \mathrm{ri}(L_\Phi)\}$,
 then
 $$
 \dim_H E_\Phi(\xi)\geq \sup\{\D(\alpha):
 \alpha\in \xi (\Sigma_A\setminus E)\cap \mathrm{ri}(L_\Phi)\}.
 $$

 \noindent(3) An extension of Theorem~\ref{main-fun-level-one-sided}(4)
  is given in the final remark of Section~\ref{appli2}.}
\end{remark}

\section{Geometric results}\label{examples}

In this section we show how the main results of the previous section
can be applied to
 multifractal analysis on conformal repellers and on attractors of conformal IFS
 satisfying the strong open set condition. Such sets  fall in
 the Moran-like geometric
 constructions considered in \cite{Bar96,P}. At first
 we describe this kind of construction (Section~\ref{gengeom}).
  Then we state
  the geometric results deduced from Theorems~\ref{main-one-sided}
  and \ref{main-fun-level-one-sided}
  (Section~\ref{maingengeom}). We give our application to {\it fixed points
    in the asymptotic average for dynamical systems} in $\R^d$ in
    Section~\ref{asympave}. Finally, we give an application
to the local scaling properties of weak Gibbs measures in
Section~\ref{appli2}, special example of which is the harmonic
measure on planar conformal Cantor sets.

\subsection{General setting of geometric realization}\label{gengeom}

  Let $(\Sigma_A,T) $ be a topologically mixing subshift of finite
type over the alphabet $\{1,\cdots,m\}$ and $\Psi\in
\C_{aa}^-(\Sigma_A,T)$. Let $X$ be $\R^{d^\prime}$ or be a
connected, $d^\prime$-dimensional $C^1$ Riemannian manifold.
Consider a family of sets $\{R_w: w\in \Sigma_{A,\ast}\}$, where
each $R_w\subset X$ is a compact set with nonempty interior. We
assume that this family of compact sets satisfies the following
conditions:

(1) $R_w\subset R_{w^\prime}$ whenever $w^\prime\prec w$.

(2)\ For any integer $n>0$, the interiors of distinct $R_w, w\in
  \Sigma_{A,n}$ are disjoint.

 (3)\ Each $R_w$ contains a ball of radius $\underline{r}_w$ and is
  contained in a ball of radius $\overline{r}_w$.

  (4)\ There exists a  constant $K>1$ and a negative sequence $\eta_n=o(n)$ such that
  for every $w\in\Sigma_{A,\ast}$,
  \begin{equation}\label{conformal-basic}
  K^{-1}\exp(\eta_{|w|})\Psi[w]\leq\underline{r}_w\leq \overline{r}_w\leq K\Psi[w].
   \end{equation}

Notice that $\Psi_{\max}<0$, then
$$
{\rm diam}(R_w)\leq 2\overline{r}_w\leq 2K\Psi[w]\leq
2K\exp(|w|\Psi_{\max})\to 0,\ \  \ (|w|\to\infty).
$$

Let $\displaystyle J=\bigcap_{n\geq 0}\bigcup_{w\in
\Sigma_{A,n}}R_w.$ We call $J$ the {\it limit set of the family
$\{R_w:w\in\Sigma_{A,\ast} \}$.} We can define the coding map $\chi:
\Sigma_A\to J$ as
$\displaystyle \chi(x)=\bigcap_{n\geq 1}R_{x|_n}, \ \ \forall x\in
\Sigma_A.$
It is clear that $\chi$ is continuous and surjective when $\Sigma_A$
is endowed with standard metric $d_1$ and $J$ is endowed with the
induced metric $\rho$ from $X$.

We say that $J$ is a {\it Moran type geometric realization of
$(\Sigma_A, d_\Psi)$.}

For this kind of construction we have the following useful
observation:

\begin{proposition}\label{dim-compare-prop}
 Let $J$ be a Moran type geometric realization of
$(\Sigma_A, d_\Psi)$ with  $\Psi\in \C_{aa}^-(\Sigma_A,T)$, then for
any $E\subset J$ we have $\dim_H E=\dim_H(\chi^{-1}(E))$.
\end{proposition}

In this paper we consider two classes of Moran type geometric
realizations of $\Sigma_A.$

{\bf (1)}\ Topologically mixing $C^1$ conformal repeller $(J,g)$. We
refer the book \cite{P} for the definitions and the basic properties
related to conformal repellers. It is well known that in this case
$(J,g)$ has a Markov partition $\{R_1,\cdots,R_m\}$. For each
$w=w_1\cdots w_n,$ define $R_w:=R_{w_1}\cap
g^{-1}(R_{w_2})\cap\cdots\cap g^{-n+1}(R_{w_n})$. Define
$\psi(x)=-\log |g^\prime(\chi(x))|$ and
$\Psi=(S_n\psi)_{n=1}^{\infty}$. By the definition of $R_w$ and the
property of Markov partition, the condition (1) and (2) are checked
directly.  (3) and (4) are stated in \cite{P} (Proposition 20.2),
except that for (4) we have an additional term
$\exp(\eta_{|w|})=\exp(-\|\Psi\|_{|w|})$ (see
Section~\ref{firstdefi} for the definition of $\|\Psi\|_{|w|}$).
This is because we only assume $g$ to be continuous rather than
H\"older continuous. Thus $J$ is a Moran type geometric realization
of $(\Sigma_A, d_\Psi)$ for some primitive matrix $A$ and the
potential $\Psi.$ Moreover in this case we have $\chi\circ T=g\circ
\chi.$

{\bf (2)} Attractors of $C^1$ conformal IFS satisfying the strong
open set condition (SOSC) (see \cite{PRSS} for details). Let
$\{f_1,\cdots, f_m\}$ be  such an IFS and denote by  $J$ its
attractor. Define $ \psi(x)=\log |f_{x_1}^\prime(\chi(Tx))|\ \
\text{and}\ \ \Psi=(S_n\psi)_{n=1}^{\infty}. $ Let $V$ be an open
set such that the SOSC holds. For $w=w_1\cdots w_n,$ define
$R_w=f_{w}(\overline{V})$, where $f_{w}:=f_{w_1}\circ \cdots\circ
f_{w_n}$. Due to the SOSC,  (1) and (2) hold for $\{R_w: w\in
\Sigma_{A,\ast}\}$. Moreover, arguments similar to those used to
prove Proposition 20.2 in \cite{P}  show that (3) and (4) also hold.
Thus, $\{R_w: w\in \Sigma_{A,\ast}\}$ is a Moran type geometric
realization of $(\Sigma_A,d_\Psi)$ with potential $\Psi$. Notice
that here $\Sigma_A$ is the full shift $\Sigma_m.$ By the uniqueness
of the attractor it is easy to verify that the attractor $J$ is the
limit set of the family $\{R_w: w\in \Sigma_{A,\ast}\}$.

\subsection{Multifractal analysis on Moran type geometric realizations}
\label{maingengeom} We are going to conduct multifractal analysis on
Moran type geometric realizations, thus we need a dynamics $g$ on
$J$ so that $(J,g)$ is a factor of some $(\Sigma_A,T)$. For $C^1$
conformal repellers, there is such a natural dynamic. For the
attractor of a $C^1$ conformal IFS, there is no such one in general,
the difficulty coming from those points having several codings.
However, under the SOSC, we can naturally define such a $g$ by
removing a ``negligible" part of $J$:

Let $\{f_1,\cdots,f_m\}$ be a $C^1$ conformal IFS satisfying the
SOSC. Let $V$ be an open set such that the SOSC holds. By
\cite{PRSS},
 such an open set always exists as soon as the mappings $f_i$ are
  $C^{1+\epsilon}$ and the OSC holds. Define
$\widetilde Z_\infty:=\bigcup_{w\in \Sigma_{A,\ast}} f_w(
\partial V)$ and $
Z_\infty:=\chi^{-1}(\widetilde Z_\infty)$. We have the following
lemma (proved in Section~\ref{sec7}):

\begin{lemma}\label{boundary}
 The set $\Sigma_{A}\setminus  Z_\infty$ is not empty and
 $\chi: \Sigma_{A}\setminus  Z_\infty\to J\setminus
 \widetilde Z_\infty$ is a bijection. Moreover
 $T(\Sigma_A\setminus Z_\infty)\subset \Sigma_A\setminus
 Z_\infty$, $T(Z_\infty)\subset Z_\infty$ and for any
 Gibbs measure $\mu$ on $\Sigma_A$ we have $\mu(Z_\infty)=0.$
\end{lemma}

By the previous lemma we can define the mapping $\widetilde g:
J\setminus \widetilde Z_\infty\to J\setminus \widetilde Z_\infty$
as
 $\widetilde g(x)=\chi \circ T\circ \chi^{-1}$.
  By construction we have $\chi\circ T= \widetilde g\circ\chi$
  over $\Sigma_{A}\setminus  Z_\infty$.

Let $J$ be a Moran type geometric realization of
$(\Sigma_A,d_\Psi)$. We set $\widetilde J=J$ when $J$ is a $C^1$
conformal repeller and $\widetilde J=J\setminus \widetilde Z_\infty$
when $J$ is the attractor of a $C^1$ conformal IFS satisfying the
SOSC.

Given a sequence of functions  $\Phi=(\phi_n)_{n=1}^\infty$ from
$\widetilde J$ to $\mathbb{R}^d$ and $\alpha \in \mathbb{R}^d$, we
set $\displaystyle E_\Phi(\alpha)=\Big \{x\in \widetilde J:
\lim_{n\to\infty}{\phi_n(x)}/{n}=\alpha\Big \}. $ We also  define $
L_\Phi=\{\alpha\in\R^d: E_\Phi(\alpha)\neq\emptyset\}. $ We must
redefine these objects because until now they were defined for
compact dynamical systems, while $\widetilde J$ may be not compact.

When $J$ is a conformal repeller the system $(J,g)$ is naturally a
TDS. For $\Phi\in \C_{aa}(J,g,d)$, if we define $\widetilde
\Phi:=(\phi_n\circ\chi)_{n=1}^{\infty}$, since $g\circ
\chi=\chi\circ T$, we have $\widetilde \Phi\in
\C_{aa}(\Sigma_A,T,d)$ with $C(\widetilde \Phi)=C(\Phi)$. And for
$\alpha\in\R^d$ we have $E_\Phi(\alpha)=\chi(E_{\widetilde
\Phi}(\alpha)).$

When $J$ is the attractor of a $C^1$ conformal IFS satisfying the
SOSC,  if $\phi$ is a continuous function from $J$ to
$\mathbb{R}^d$, it generates the additive potential
$\widetilde\Phi=(S_n\tilde\phi)_{n=1}^\infty$ on $(\Sigma_A,T)$,
where $\tilde\phi=\phi\circ\chi$, and  it also defines
 $\Phi=(S_n\phi)_{n=1}^\infty$ on $(\widetilde J,\widetilde g)$.  Then for
 $\alpha\in\R^d$ we have
$ E_\Phi(\alpha)=\chi (E_{\widetilde \Phi}(\alpha)\setminus
Z_\infty). $

Write $\D_\Phi(\alpha):=\dim_H E_\Phi(\alpha)$ for convenience.

\begin{theorem}\label{appli-one-sided}
Let $J$ be a Moran type geometric realization of
$(\Sigma_A,d_\Psi)$. If $J$ is a $C^1$ conformal repeller, let
$\Phi\in \C_{aa}(J,g,d)$ and define $\widetilde{\Phi}$ as above. If
$J$ is the attractor of a $C^1$ conformal IFS satisfying the SOSC,
let $\phi$ be a continuous map from $J$ to $\R^d$, and define the
additive potential $\widetilde\Phi=(S_n\tilde\phi)_{n=1}^\infty$ on
$(\Sigma_A,T)$ with $\tilde\phi=\phi\circ\chi$ and
$\Phi=(S_n\phi)_{n=1}^\infty$ on $(\widetilde J,\widetilde g)$. Then
\begin{enumerate}
\item $L_\Phi=L_{\widetilde\Phi}$;  for  $\alpha\in L_{\Phi}$ we have
 $\D_\Phi(\alpha)=\dim_PE_\Phi(\alpha)$ and
$$
\D_\Phi(\alpha)=\D_{\widetilde \Phi}(\alpha)=
\Lambda_{\widetilde{\Phi}}(\alpha)=\mathcal{E}_{\widetilde{\Phi}}(\alpha).$$

\item  $\dim_H J=\dim_B J=D(\Psi)=\max\{\D_\Phi(\alpha): \alpha\in L_{{\Phi}}\}$.
\end{enumerate}
\end{theorem}

\begin{remark}{\rm
For the case of conformal repellers, the connection between
Theorem~\ref{appli-one-sided} and the other works \cite{BSS,FLW,BD}
is similar to that done in Remark~\ref{connexion}(1) and (2).}
\end{remark}

For the set $E_\Phi(\xi)$ we have the following result:

\begin{theorem}\label{appli-fun-level-one-sided}
Let $J$ be a Moran type geometric realization of
$(\Sigma_A,d_\Psi)$, which is either a $C^1$ conformal repeller or
the attractor of a $C^1$ conformal IFS satisfying the SOSC. Let
$\Phi$ and $\widetilde \Phi$ be the same as in Theorem
\ref{appli-one-sided}. Let $\xi:J\to\R^d$ be continuous and
$\displaystyle E_\Phi(\xi)=\Big\{x\in\widetilde J:
\lim_{n\to\infty}{\phi_n(x)}/{n}=\xi(x)\Big \}. $ If $\xi(J)\subset
\mathrm{aff}(L_{\Phi})$, then

\begin{enumerate}
\item  $\dim_H E_\Phi(\xi)\geq \sup\{\D_\Phi(\alpha):\alpha\in \xi(
J)\cap \mathrm{ri}(L_{{\Phi}})\},$ and $E_\Phi(\xi)$ is dense if
$\xi(J)\subset L_\Phi$.

\item  If  $\sup\{\D_\Phi(\alpha):\alpha\in \xi(J)\cap
\mathrm{ri}(L_{{\Phi}})\}= \sup\{\D_\Phi(\alpha):\alpha\in
\xi(J)\cap L_{{\Phi}}\}$, then
$$
{\dim}_H E_\Phi(\xi)=\dim_P
E_\Phi(\xi)=\sup\{\D_\Phi(\alpha):\alpha\in \xi(J)\cap L_{{\Phi}}\}.
$$

\item If $d=1$ and $\xi(J)\subset L_\Phi$, then $E_\Phi(\xi)$ is dense and
 $$
 {\dim}_H
E_\Phi(\xi)=\dim_P E_\Phi(\xi)=\sup\{\D_\Phi(\alpha):\alpha\in
\xi(J)\}.
$$
\end{enumerate}
\end{theorem}

\subsection{Application to fixed points
in the asymptotic average for dynamical systems in $\R^d$}
\label{asympave}

Suppose that $(J,g)$ is a dynamical system with $J\subset \R^d$.
 We say that $x\in J$ is a {\it fixed point of $g$
in the asymptotic average} if $\displaystyle \lim_{n\to\infty}
\frac{1}{n}\sum_{k=0}^{n-1} g^k x=x$.
 We are interested in the Hausdorff dimension of the set of all  such  points:
 $$
\mathcal{F}(J,g)=\Big\{x\in J: \lim_{n\to\infty}\frac{1}{n}
\sum_{k=0}^{n-1} g^k x=x\Big \}.
$$
If $\xi$ stands for the identity map on $J$ and $\Phi$ stands for
the additive potential
 associated with the potential $\xi$, in our setting we have
 $\mathcal{F}(J,g)=E_\Phi(\xi)$.

The set $L_\Phi$ is contained in the convex hull of $J$, and it
contains the set  of the fixed points of $g$.  An example of trivial
situation is provided by the unit circle endowed with dynamic
$g(z)=z^2$ in $\mathbb{C}$. There, $\mathcal{F}(J,g)=\{1\}$. How
about general conformal repellers and attractors of conformal IFS?
This question is non trivial in general.  We are  going to describe
a class of conformal IFS, namely self-similar generalized Sierpinski
carpets, for which the situation is non trivial and we have a
complete answer.

We consider a special self-similar IFS $\{f_1,\cdots,f_m\}$ on
$\R^d$:
 $f_j(x)=\rho_jx+c_j,  \ 0<\rho_j<1,  \ (1\leq j\leq m)$.
 We assume further the SOSC fulfills. Let $x_j$ stand for the unique fixed point of
 $f_j$ and let $J$ be the attractor  of this IFS.
 Notice that the
 mappings $f_j$ have no rotation part, thus the convex hull
 of $J$ satisfies ${\rm Co}(J)={\rm Co}\{x_1,\cdots,x_m\}=:\Delta$,
 and is a convex polyhedron. We further
 assume that ${\rm Co}(J)$ has dimension $d$ (otherwise we can define this IFS in
 a smaller
 affine subspace).

 Let $W$ stand for the open set such that the SOSC
 holds. It is ready to see that $V:=W\cap \Delta$ is also an open
 set such that SOSC holds. We can define the
 dynamics $\widetilde{g}$ on $\widetilde{J}=J\setminus \widetilde Z_\infty$,
  where $\widetilde Z_\infty$ is defined as in the previous subsection.

Now we have the following result whose proof is given in
Section~\ref{sec7}.

 \begin{theorem}\label{carpet}
  Let $\Phi={\rm id}_J.$ Then ${\mathcal F}(\widetilde{J},\widetilde{g})$
  is dense and
  $\dim_{H}{\mathcal F}(\widetilde{J},\widetilde{g})=
  \sup\{\D_\Phi(\alpha): \alpha\in J\}. $
  Moreover if the point at which $\D_\Phi$ attains its maximum belongs to $J$,
  then ${\mathcal F}(\widetilde{J},\widetilde{g})$ is of full Hausdorff dimension.
 \end{theorem}
We have the following corollary, in which the lower bound for
 the Hausdorff dimension follows directly from Theorem~\ref{carpet}
 and the upper bound follows from standard estimates based on the
 bounds provided in Section~\ref{upperbd}.
\begin{corollary}
Let $N\in \mathbb N_+$ and  let $d_1,\dots d_N$ be $N$ positive
integers. Consider $N$
 self-similar IFS without rotations components
 $\{f_1^{(j)},\cdots,f_{m_j}^{(j)}\}_{1\le j\le N}$, satisfying
 SOSC and
 living respectively in $\R^{d_j}$. Denote by $J_j$, $1\le j\le N$,
 their respective attractors as
  well as the corresponding dynamical systems $(\widetilde J_j,\widetilde g_j)$.
   Let $\widetilde J=\prod_{j=1}^N \widetilde J_j\subset \R^{\sum_{j=1}^Nd_j}$
   be endowed with the dynamics $\widetilde g=(\widetilde g_1,\dots,\widetilde g_N)$.
   We have $\dim_H\mathcal{F} (\widetilde J,\widetilde g)=
   \sum_{j=1}^N \dim_H\mathcal{F}
   (\widetilde J_j,\widetilde g_j)=\sum_{j=1}^N\sup\{\D_{\Phi_j}(\alpha): \alpha\in
   J_j\}$, where $\Phi_j={\rm Id}_{\R^{d_j}}$.
\end{corollary}

Both the previous results yield the result presented in the
introduction of the paper:
\begin{theorem}
Let $d\in\N$ and $(m_1,\dots,m_d)$ be $d$ integers $\ge 2$. Set
$J=[0,1]^d$ and let $g:J\to J$ be the mapping
$(x_1,\dots,x_d)\mapsto (m_1x_1\pmod 1,\dots,m_d x_d\pmod 1)$. Then
$\mathcal{F}(J,g)$ is dense and of full Hausdorff dimension in
$[0,1]^d$.
\end{theorem}
To see this, for a fixed integer $m\ge 2$ let $g_m:[0,1]\to[0,1]$ be
the mapping $x\mapsto m x\pmod 1.$

Let $(\Sigma_m,T)$ be the full shift over alphabet
$\{0,\cdots,m-1\}$, where   $\Sigma_m$ is endowed with the usual
metric  $d_\Psi(x,y)=d_1(x,y)=m^{-|x\wedge y|}. $  Define a map
$\chi:\Sigma_m\to [0,1]$ as $ \chi(x)=\sum_{n=1}^\infty {x_n}/{m^n}.
$ Then $\chi$ is continuous and surjective. Consider the IFS
$\{f_j:j=0,\cdots,m-1\}$ defined as $f_j(x)=(x+j)/m$. It is seen
that the SOSC holds with $V=(0,1)$.
 Let
$ \widetilde Z_\infty:=\Big \{\sum_{j=1}^n x_jm^{-j}: n\in\N;
x_j=0,\cdots,m-1\Big \}\cup \{1\}. $ Define the dynamics
$\widetilde{g}$ on $\widetilde J_m=[0,1]\setminus\widetilde
Z_\infty$ as in the previous section. Then it is easy to check that
$\widetilde{g}={g_m}|_{\widetilde J_m}.$ Let $\Phi={\rm
id}_{[0,1]}$.  By theorem \ref{carpet} we get $ \dim_H
\mathcal{F}(\widetilde
J_m,g_m)=\sup\{\D_{{\Phi}}(\alpha):\alpha\in[0,1]\}. $ By the law of
large number applied to the measure of maximal entropy we get
$\D_\Phi(1/2)=1.$ We conclude by noticing that
$\mathcal{F}(J,g)=\prod_{i=1}^d\mathcal{F}([0,1],g_{m_i})$.

Next we consider concrete examples of carpets in the unit square.

\subsubsection*{{\bf  Heterogeneous carpets in the unit square}}
In order to fully illustrate our purpose, we consider an IFS
$S_0=\{f_1,\cdots,f_N\}$ in $\R^2$ made of contractive similitudes
without rotations such that the squares $f_i([0,1]^2)$ form a tiling
of $[0,1]^2$. All these situations have been determined in
\cite{Brooks}.
 In this way, $]0,1[^2$ can be chosen as the open set such that the SOSC holds,
  and the boundaries of the sets $f_i(]0,1[^2)$ have big intersections.
 The  picture on the left of Figure 1 give an example of this kind of
IFS. This IFS contains $15$ dilation maps, and the dynamics on this
attractor is highly non trivial.

 Let $\Phi$ denote ${\rm Id}_{\R^2}$. For each $\emptyset\neq S\subset S_0$,
  we denote by $J_S$ the attractor of the IFS $S$.
 The dynamics $\widetilde g_S$ defined on $\widetilde J_S$ is
 the restriction  of  $\widetilde g_{S_0}$
 to $\widetilde J_{S}$. The set $\mathcal{F}(\widetilde J_{S_0},\widetilde g_{S_0})$
  is of full Hausdorff dimension,
  since $J_{S_0}=[0,1]^2$. If $S\neq S_0$, we have the variational
  formula $\dim_H \mathcal{F}
  (\widetilde J_{S},\widetilde g_{S})=\sup\{\D_{\Phi}(\alpha):\alpha\in J_S\}$,
   and in general it
  is hard to know whether $ \mathcal{F}(\widetilde J_{S},\widetilde g_{S})$
  is of full dimension or not
  in $J_S$. However, here are two  simple examples illustrating both possibilities.

 We consider the case of the regular tiling associated with the IFS
 $ S_0=\Big \{f_{i,j}:x\mapsto
 \frac{x}{3}+\frac{(i,j)}{3}:0\le i,j\le 2\Big \}$. Then,
 let $S_1=\{f_{0,0},f_{0,2},f_{2,0},f_{2,2}\}$ and $S_2=S_1\cup\{f_{1,1}\}$.
We claim that $\mathcal{F}(\widetilde J_{S_1},\widetilde g_{S_1})$
is not of full Hausdorff dimension,
 while $\mathcal{F}(\widetilde J_{S_2},\widetilde g_{S_2})$ is.

The simpler situation is that of $S_2$. In this case, $G=(1/2,1/2)$,
the center of symmetry of $J_{S_2}$
 is the fixed point of $f_{1,1}$ and it belongs to $L_{\Phi}$. Moreover,
 it is obvious that the uniform measure
 (or Parry measure) on $J_{S_2}$ is carried by the
  set $E_{\Phi}(G)$. This yields the result by Theorem~\ref{carpet},
  and $\dim_H \mathcal{F}(\widetilde J_{S_2},\widetilde g_{S_2})=\log 5/\log 3$.

In the case of $S_1$, the point $G$ is still the center of symmetry
of $J_{S_1}$, so  $\mathcal D_{\Phi}$ reaches its maximum at $G$.
However, $G$ does not belong to $J_{S_1}$.  Since $\Phi$ is H\"older
continuous and the tiling is regular, we know that $\mathcal
D_{\Phi}$ is strictly concave. By using the symmetry, one deduces
that the restriction of $\mathcal D_{\Phi}$ to $J_{S_1}$ reaches its
maximum at any of the four points $(1/3,1/3)$, $(1/3,2/3)$,
$(2/3,1/3)$ and $(2/3,2/3)$. This yields $\displaystyle \dim_H
\mathcal{F}(\widetilde J_{S_1},\widetilde
g_{S_1})=\D_\Phi((1/3,1/3))<\log 4/\log 3=\dim_H J_{S_1}$.

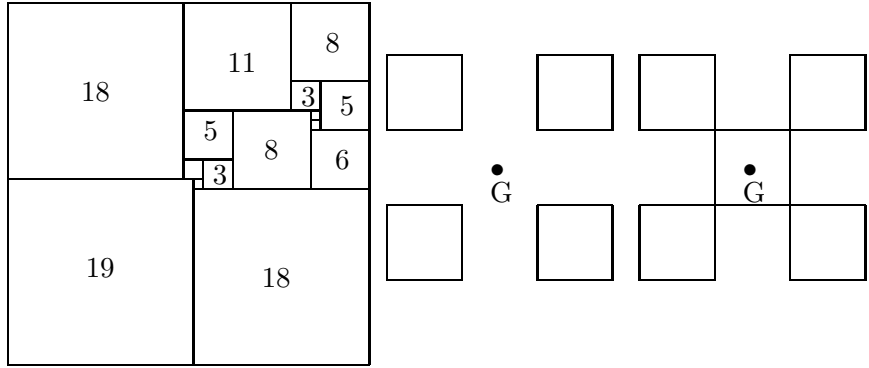
\begin{figure}[h]
\setlength{\unitlength}{0.28cm}
\begin{picture}(36,18)(0,-1)

\put(16,2){
\begin{picture}(9,9)\setlength{\unitlength}{0.5cm}
\multiput(0,0)(0,4){2}{\line(0,1){2}}
\multiput(2,0)(0,4){2}{\line(0,1){2}}
\multiput(4,0)(0,4){2}{\line(0,1){2}}
\multiput(6,0)(0,4){2}{\line(0,1){2}}
\multiput(0,0)(4,0){2}{\line(1,0){2}}
\multiput(0,2)(4,0){2}{\line(1,0){2}}
\multiput(0,4)(4,0){2}{\line(1,0){2}}
\multiput(0,6)(4,0){2}{\line(1,0){2}}
\put(2.75,2.75){$\bullet$} \put(2.75,2.1){G}
\end{picture}
}

\put(28,2){
\begin{picture}(9,9)\setlength{\unitlength}{0.5cm}
\multiput(0,0)(0,4){2}{\line(0,1){2}}
\multiput(2,0)(0,4){2}{\line(0,1){2}}
\multiput(4,0)(0,4){2}{\line(0,1){2}}
\multiput(6,0)(0,4){2}{\line(0,1){2}}
\multiput(0,0)(4,0){2}{\line(1,0){2}}
\multiput(0,2)(4,0){2}{\line(1,0){2}}
\multiput(0,4)(4,0){2}{\line(1,0){2}}
\multiput(0,6)(4,0){2}{\line(1,0){2}}
\multiput(2,2)(0,2){2}{\line(1,0){2}}
\multiput(2,2)(2,0){2}{\line(0,1){2}}
\put(2.75,2.75){$\bullet$} \put(2.75,2.1){G} 
\end{picture}
}

\put(-2,-2){
\begin{picture}(9,9)\setlength{\unitlength}{0.13cm}
\multiput(0,0)(0,37){2}{\line(1,0){37}}\multiput(0,0)(37,0){2}{\line(0,1){37}}
\multiput(0,0)(0,19){2}{\line(1,0){19}}\multiput(0,0)(19,0){2}{\line(0,1){19}}\put(8,9){19}
\multiput(19,0)(0,18){2}{\line(1,0){18}}\multiput(19,0)(18,0){2}{\line(0,1){18}}\put(26,8){18}
\multiput(19,18)(0,1){2}{\line(1,0){1}}\multiput(19,18)(1,0){2}{\line(0,1){1}}
\multiput(20,18)(0,3){2}{\line(1,0){3}}\multiput(20,18)(3,0){2}{\line(0,1){3}}\put(21,18.5){3}
\multiput(23,18)(0,8){2}{\line(1,0){8}}\multiput(23,18)(8,0){2}{\line(0,1){8}}\put(26.25,21){8}
\multiput(31,18)(0,6){2}{\line(1,0){6}}\multiput(31,18)(6,0){2}{\line(0,1){6}}\put(33.5,20){6}
\multiput(0,19)(0,18){2}{\line(1,0){18}}\multiput(0,19)(18,0){2}{\line(0,1){18}}\put(7.5,27){18}
\multiput(18,19)(0,2){2}{\line(1,0){2}}\multiput(18,19)(2,0){2}{\line(0,1){2}}
\multiput(18,21)(0,5){2}{\line(1,0){5}}\multiput(18,21)(5,0){2}{\line(0,1){5}}\put(20,23){5}
\multiput(18,26)(0,11){2}{\line(1,0){11}}\multiput(18,26)(11,0){2}{\line(0,1){11}}\put(22.5,30){11}
\multiput(31,24)(0,1){2}{\line(1,0){1}}\multiput(31,24)(1,0){2}{\line(0,1){1}}
\multiput(31,25)(0,1){2}{\line(1,0){1}}\multiput(31,25)(1,0){2}{\line(0,1){1}}
\multiput(32,24)(0,5){2}{\line(1,0){5}}\multiput(32,24)(5,0){2}{\line(0,1){5}}\put(34,25.5){5}
\multiput(29,29)(0,8){2}{\line(1,0){8}}\multiput(29,29)(8,0){2}{\line(0,1){8}}\put(32.5,32){8}
\put(30,26.5){3}
\end{picture}
}
\end{picture}
\bigskip
\caption{{\bf Left:} Example of tiling of $[0,1]$ by squares.
 {\bf Middle:} IFS $S_1=\{f_{0,0},f_{0,2},f_{2,0},f_{2,2}\}$.
 {\bf Right:} IFS $S_2=S_1\cup\{f_{1,1}\}$.}
\end{figure}

\subsection{Localized results for weak Gibbs measures}\label{appli2}
Let $\{f_1,\cdots,f_m\}$ be a homogenous self-similar IFS in
$\mathbb C$ satisfying the {\it  strong separation condition}, that
is, each function $f_j$ has the form $f_j(z)=a_jz+b_j$ where
$0<\rho=|a_j|<1,$ and there exists a topological closed  disk $D$
such that $f_j(D)\subset D$  and  the $f_j(D)$ are pairwise
disjoint. There is a natural coding map $\chi:\Sigma_m\to J$.
Moreover if we define $\psi(x)\equiv\log \rho$ for $x\in\Sigma_m$,
and $\Psi=(S_n\psi)_{n=1}^{\infty}$, then
$\chi:(\Sigma_m,d_\Psi)\to(J,|\cdot|)$ is a bi-Lipschitz
homeomorphism.

 Let $\phi: J\to \R$ be  continuous and define $\tilde \phi=\phi\circ
 \chi.$ By subtracting a constant potential if necessary,
  we can assume $P(T,\tilde\phi)=0.$ There
exists a weak Gibbs measure $\tilde\mu$ on $\Sigma_m$ (see
\cite{K}), i.e. a probability measure
 such that the exists positive sequence $(C_n)_{n\ge 1}$ such
 that for all $x\in \Sigma_m$ and $n\ge 1$
 $$
C_n^{-1} \exp(S_n\tilde \phi(x))\le \widetilde \mu([x|_{n}])\le
 C_n\exp(S_n\tilde \phi(x)),
$$
with $\lim_{n\to\infty} \log(C_n)/n=0$ (if $\phi$ is H\"older
continuous, then $C_n$ is bounded and $\widetilde \mu$ is a Gibbs
measure). In particular,
 $$d_{\tilde\mu}(x):=\lim_{r\to 0}\frac{\log \widetilde\mu(B(x,r))}{\log
 r}=\lim_{n\to\infty}\frac{S_n\tilde\phi(x)}{n\log\rho}$$
 in the sense that either both the limits do not exist, either they
 exist and are equal. Define $\mu:=\chi_\ast(\tilde\mu)$ and $\mu$ is
 called a {\it weak Gibbs measure} associated with $\phi.$
  By the bi-Lipschitz property of $\chi$ and the strong separate condition, we can
 easily conclude that
 $d_\mu(y)=\lim_{n\to\infty}S_n\phi(y)/(n\log\rho)$ for any $y\in
 J.$ Let $\Phi=(S_n\phi)_{n=1}^\infty.$ If we define
  $E_\mu(\alpha)=\{y\in J: d_\mu(y)=\alpha\}$, then
 we get $E_\Phi(\alpha)=E_\mu(\alpha/\log \rho)$ for any $\alpha\in L_\Phi.$

 By  applying Theorem
\ref{appli-fun-level-one-sided} for $d=1$, we have the following
property regarding the local property of
 weak Gibbs measure:

 \begin{corollary}\label{weakgibbs}
Let $\mu$ be the weak Gibbs measure associated with $\phi.$ Then the
set of all possible local dimension for $\mu$ is the interval $
L_\Phi/\log\rho.$ Assume $\xi:J\to\R$ is
 continuous and $\xi(J)\subset  L_\Phi/\log\rho$,
 then
 $$
\dim_H\{x\in J: d_\mu(x)=\xi(x)\}=\sup\{\dim_H
E_\mu(\alpha):\alpha\in\xi(J)\}.
 $$
 \end{corollary}

 Now let $\omega$ stand for the harmonic measure on $J$. It is well
 known that (see for example the survey paper \cite{Ma}) there
 exists a H\"older continuous function $\phi:J\to\R$ such that $w\asymp
 \mu$, where $\mu$ is the equilibrium state of $\phi.$

 \begin{corollary}\label{harmonic}
Let $\omega$ be the harmonic measure on $J$ and $I$ is the set of
all possible local dimension for $\omega$.  Assume $\xi:J\to\R$ is
 continuous and $\xi(J)\subset I$.
 Then
 $$
\dim_H\{x\in J: d_\omega(x)=\xi(x)\}=\sup\{\dim_H
E_\omega(\alpha):\alpha\in\xi(J)\}.
 $$
 \end{corollary}

\noindent {\bf Final remark.} At least when $d=1$, it is not
difficult to extend the
 results obtained in
 this paper by considering $\Upsilon=(\gamma_n)_{n\ge 1}\in \C_{aa}^+(\Sigma_A,T)$
 and the more general level sets $E_{\Phi/\Upsilon}(\xi)=\{x\in \Sigma_A:
  \lim_{n\to\infty}\phi_n(x)/\gamma_n(x)=\xi(x)\}$;
  when $\xi$ is constant, such sets have been considered in the contexts
  examined in \cite{BSS,BD}.
   The formula is that if the continuous function $\xi$ takes values in the set
   $L_{\Phi/\Upsilon}=\{\Phi_*(\nu)/\Upsilon_*(\nu):\nu\in\mathcal{M}(\Sigma_A,T)\}$,
   then $\dim_H (E_{\Phi/\Upsilon}(\xi))=\sup\{-h_\nu(T)/\Psi_*(\nu):
   \nu\in \mathcal{M}(\Sigma_A,T),\ \Phi_*(\nu)/\Upsilon_*(\nu)\in \xi(\Sigma_A)\}$.
   When $\Upsilon =-\Psi$, this can be applied to the local dimension
   of Gibbs measures
    associated with H\"older potentials $\varphi$ on any $C^1$ conformal
    repeller of a map $f$,
    since in this case we know from \cite{P} that such a measure is doubling
    so that the
     local dimension is directly related to the asymptotic behavior of
      $S_n\varphi/S_n(-\log \|Df\|)$.
      Consequently, Corollary~\ref{harmonic} can be extended to harmonic
       measure on more general conformal repellers (see \cite{Ma}).

\subsection*{Additional definitions and notations}

For $\Phi\in\C_{aa}(X,T)$ (see \eqref{aa}) recall that we defined
$\Phi_{\max}:=\max(\phi_{1})+C(\Phi) $ and
$\Phi_{\min}:=\min(\phi_{1})-C(\Phi)$. Then define
$\|\Phi\|:=|\Phi_{\max}|\vee|\Phi_{\min}|.$ By the almost additivity
property we easily get
\begin{equation}\label{max-min}
n\Phi_{\min}\leq \phi_n(x)\leq n\Phi_{\max}, \ \ \forall\  n\in\N.
\end{equation}
Consequently we have $\|\phi_n\|_\infty\leq n\|\Phi\|$.

If $\Phi=(\Phi^1,\cdots,\Phi^d)\in\C_{aa}(X,T,d)$, we define
$\Phi_{\max}:=(\Phi^1_{\max},\cdots,\Phi^d_{\max})$ and
$\Phi_{\min}:=(\Phi^1_{\min},\cdots,\Phi^d_{\min})$. We also define
$\displaystyle \|\Phi\|:=\Big(\sum_{j=1}^{d}\|\Phi^j\|^2\Big)^{1/2}$
and $\|\Phi\|_{\mbox{\tiny \rm
lim}}:=\limsup_{n\to\infty}\frac{\|\phi_n\|_\infty}{n}. $ We have
 $\|\phi_n\|_\infty\leq n\|\Phi\|$.

Given  $u,v\in \R^d$, we write $[u,v]:=\{tu+(1-t)v: 0\leq t\leq 1\}$
to denote the closed interval connecting $u$ and $v$.  If $u_i\leq
v_i$ for $i=1,\cdots,d$, then we write $u\leq v.$ If
$u,v_1,v_2\in\R^d$ is such that $v_1\leq u\leq v_2$, it is easy to
prove that $|u|\leq |v_1|+|v_2|$.  We will use this basic fact
several times.

 For
$\Phi\in\C_{aa}(X,T,d)$, after defining
$C(\Phi):=(C(\Phi^1),\cdots,C(\Phi^d))$, we also have the following
vector almost additivity:
$$
-C(\Phi)+ \phi_n+\phi_p\circ T^n\leq \phi_{n+p}\leq C(\Phi)+
\phi_n+\phi_p\circ T^n,\quad \forall\  n,p\in \N.
$$

The simplest almost additive potentials are the additive ones: Given
$\phi:X\to\R^d$ continuous, define
$\phi_n=S_n\phi:=\sum_{j=0}^{n-1}\phi\circ T^j$.   If $\phi$ is
H\"{o}lder continuous, we also say that
$\Phi=(S_n\phi)_{n=1}^{\infty}$
 is {\it H\"{o}lder continuous}. The simplest
H\"{o}lder continuous potentials are the constant potentials
$(n\alpha)_{n=1}^\infty$, $\alpha\in\R^d$, that we also denote
as~$\alpha.$

\section{Proofs of Propositions~\ref{dim-formular-1}
and \ref{regularity}}\label{proofs}

\subsection{Proof of Proposition ~\ref{dim-formular-1}}
We need some facts gathered in the two following lemmas.
 We omit their simple proofs based on elementary using
 of the almost additivity of $\Phi$ and the continuity of the $\phi_n$.

\begin{lemma}\label{multiplictive}

\begin{enumerate}
\item  Given $\Phi\in \C_{aa}(\Sigma_A,T,d)$ and two constants $C_2\geq
C_1>0$, for each $n\in \N$ define
\begin{equation}\label{variation}
\|\Phi\|_{n}^\star:=\max\{\|\Phi\|_l:C_1n\leq l\leq C_2n\}.
\end{equation}
Then  $\|\Phi\|_{n}^\star/n\to 0$ when $n\to\infty.$ Especially
$\lim_{n\to\infty}\|\Phi\|_n/n=0$.

\noindent Let  $\Phi\in \C_{aa}(\Sigma_A,T)$ and  $C=C(\Phi)$.

\item  For any $u,v\in \Sigma_{A,\ast} $ such that $uv\in  \Sigma_{A,\ast}$ we have
$$ \exp(-C-\|\Phi\|_{|u|})\Phi[u]\Phi[v]\leq \Phi[uv]\leq \exp(C)\Phi[u]\Phi[v].$$

\item  For $w=u_1w_1\cdots u_nw_nu_{n+1}\in\Sigma_{A,\ast}$, let
$k=\sum_{j=1}^{n+1}|u_j|$. We have
\begin{equation}\label{222}
\exp(-2nC+k\Phi_{\min})\prod_{j=1}^{n}\Phi[w_j]\exp(-\|\Phi\|_{|w_j|})
\leq \Phi[w]\leq \exp(2nC+k\Phi_{\max})\prod_{j=1}^{n}\Phi[w_j].
\end{equation}

\item If  $\Phi\in \C_{aa}^-(\Sigma_A,T)$, then $\Phi[v]\leq\Phi[u]$ for
$u\prec v$.
\end{enumerate}
\end{lemma}

\begin{lemma} \label{Moran-covering}
Let $\Psi\in  \C_{aa}^-(\Sigma_A,T).$
\begin{enumerate}
\item Let $C_1(\Psi)=1/|\Psi_{\min}|$ and
$C_2(\Psi)=1+1/|\Psi_{\max}|.$ For any $w\in{\mathcal B}_n(\Psi)$ we
have
\begin{equation}\label{constant}
C_1(\Psi)n\leq|w|\leq C_2(\Psi)n.
\end{equation}

\item For any $w\in{\mathcal B}_n(\Psi)$ we have
\begin{equation}\label{444}
\exp(-C(\Psi)-\|\Psi\|_{|w|}+\Psi_{\min})e^{-n}\leq\Psi[w]\leq
e^{-n}.
\end{equation}

\item The balls in  $\{[w]: w\in{\mathcal B}_n(\Psi)\}$ are pairwise
disjoint.

\item  If $u\prec v$ are such that $u\in {\mathcal B}_{n_1}(\Psi)$ and
$v\in {\mathcal B}_{n_2}(\Psi)$, then
$$|v|-|u|\leq \frac{\Psi_{\min}-\|\Psi\|_{|v|}-(n_2-n_1)-2C(\Psi)}{\Psi_{\max}}.$$
\end{enumerate}
\end{lemma}

Let us start the proof of Proposition \ref{dim-formular-1}. The hard
part is $(\ref{large-div}).$ At first we will show that $\log
f(\alpha,n,\epsilon)$, as a sequence of $n$, has a kind of
subadditivity property. Due to this subadditivity, by a standard
procedure, we get the desired equality of the two limits. The proof
is an adaption of that given in \cite{FFW} (see Proposition 5) and
\cite{FLW} (see Proposition 4.3). However instead of $d_1$ and an
additive potential $\phi$ considered in \cite{FFW}, here we consider
$d_\Psi$ and an almost additive potential $\Phi$, so the proof is
more involved.

\noindent $\bullet$\ {\it Subadditivity of $\log
f(\alpha,n,\epsilon)$.} More precisely we will show that for any
$\epsilon>0$, there exist an $N\in \N$ and positive sequence
$\{\beta_n\}$ with $\log\beta_n=o(n)$  such that
$$
f(\alpha,n,\epsilon)^p\leq
\beta_n^pf(\alpha,(n+\widetilde{c})p,2\epsilon)$$ for any $ n\geq
N,$ and any $ p\geq 1,$
where $\widetilde{c}=[-p_0\Psi_{\max}-2C(\Psi)]$. Recall that $p_0$
is a fixed positive integer such that $A^{p_0}>0.$

 To each $(w_1,\cdots,w_p)\in F(\alpha,n,\epsilon)^p$ we
 can associate an element of
 $F(\alpha,(n+\widetilde{c})p,2\epsilon)$ as follows. Let
$w=\overline{w}_1\cdots \overline{w}_p$, where
$\overline{w}_j=w_ju_j$ with $u_j\in \mathcal{W}$ such that
$w_ju_jw_{j+1}$ is admissible, where $\mathcal{W}$ is defined in
$(\ref{big-xi})$. Recall (see \eqref{normn}) that for any cylinder
$[u]$ and any $x,\widetilde{x}\in [u]$, we have
$|\psi_{|u|}(x)-\psi_{|u|}(\widetilde{x})|\leq \|\Psi\|_{|u|}.$
Thus for any $x\in[u],$
\begin{equation}\label{111}
\exp(\psi_{|u|}(x))\geq \Psi[u]\exp(-\|\Psi\|_{|u|}).
\end{equation}

Now for any $x\in [w]$,  let $ s_0=0, s_k=\sum_{j=1}^{k}(|w_j|+p_0)\
(1\leq k\leq p)$ and define $x^{k}=T^{s_{k-1}}x$. We have
$|w|=s_{p}$ and $x^k\in[w_k]$ for $k=1,\cdots,p.$
Then, by using the almost additivity of $\Psi$, $(\ref{111})$ and
Lemma \ref{Moran-covering}(2)
 we can get
\begin{multline*}
\Psi[w]\geq \exp(\psi_{|w|}(x))
\geq\exp\Big(\sum_{k=1}^p \psi_{|w_k|}(x^k)+p_0p\Psi_{\min}-(2p-1)C(\Psi)\Big)\\
\geq (\prod_{k=1}^{p}\Psi[w_k])\exp\Big (-\sum_{k=1}^p
\|\Psi\|_{|w_k|}+p(p_0\Psi_{\min}-2C(\Psi))\Big )\ge
\exp(-p(n+c_1(n))),
\end{multline*}
where  $c_1(n)=-(p_0+1)\Psi_{\min}+3C(\Psi)+2\|\Psi\|_{n}^\star>0$
and $\|\Psi\|_{n}^\star$ is defined as in $(\ref{variation})$ with
the constants $C_2(\Psi)\geq C_1(\Psi)>0$. Lemma \ref{multiplictive}
(1) yields $c_1(n)/n\to 0.$ By Lemma \ref{multiplictive}(3) and the
definition of $\tilde c$ we also have
\begin{eqnarray*}
\Psi[w]&\leq&
\exp(2pC(\Psi)+p_0p\Psi_{\max})(\prod_{k=1}^{p}\Psi[w_k]) \\
&\leq& \exp(-p(n-p_0\Psi_{\max}-2C(\Psi))\leq \exp(-p(n+\tilde c)).
\end{eqnarray*}
Thus there exists $u\in {\mathcal B}_{p(n+\widetilde{c})}(\Psi)$
such that $u\prec w$. Write ${w}=uw^\prime$.

\noindent {\bf Claim:}\ \ $|w^\prime|\leq p(ac_1(n)+b)$ for some
constant $a,b>0.$

Indeed, we have
$$
e^{-p(n+c_1(n))}\leq\Psi[{w}]\leq
e^{C(\Psi)}\Psi[u]\Psi[w^\prime]\leq
 e^{C(\Psi)}e^{-p(n+\tilde c)} e^{|w^\prime|\Psi_{\max}}.
 $$
Thus $|w^\prime|\leq p (c_1(n)+C(\Psi)-\tilde c)/(-\Psi_{\max})\leq
p (c_1(n)+3C(\Psi)+1)/(-\Psi_{\max}).$

Now since $w_k\in F(\alpha,n,\epsilon)$ we can find
$x_{k}\in[w_{k}]$ such that
$|\frac{\phi_{|w_k|}(x_k)}{|w_k|}-\alpha|<\epsilon.$ Take $x\in
[w]$; in particular, $x\in[u]$. Define $s_k$ and $ x^{k}$ as above.
We have $|{w}|=s_{p}$ and $x^k\in[w_k]$ for $k=1,\cdots,p.$ By
almost additivity, we get
$$
\phi_{|u|}(x)+\phi_{|w^\prime|}(T^{|u|}x)-C(\Phi)\leq\phi_{|w|}(x)\leq
\phi_{|u|}(x)+\phi_{|w^\prime|}(T^{|u|}x)+C(\Phi)
$$
(this is a vector inequality). Recall  that if $\beta,
\beta_1,\beta_2\in \R^d$ are such that $\beta_1\leq \beta\leq
\beta_2$, then $|\beta|\leq |\beta_1|+|\beta_2|.$ Then we conclude
that
\begin{eqnarray*}
\Big|\phi_{|w|}(x)-\phi_{|u|}(x)\Big|&\leq&
\Big|\phi_{|w^\prime|}(T^{|u|}x)-C(\Phi)\Big|+
\Big|\phi_{|w^\prime|}(T^{|u|}x)+C(\Phi)\Big|\\
&\leq& 2(|w^\prime|\|\Phi\|+|C(\Phi)|).
\end{eqnarray*}
In other word, $\phi_{|u|}(x)=\phi_{|w|}(x)+\eta_0$ with
$|\eta_0|\leq 2|w^\prime|\|\Phi\|+2|C(\Phi)|.$ In a similar fashion
we have
\begin{eqnarray*}
\phi_{|u|}(x)&=&\phi_{|w|}(x)+\eta_0=\sum_{k=1}^{p}\phi_{|w_ku_k|}(x^k)
+\eta_1+\eta_0\\
&=&\sum_{k=1}^{p}\phi_{|w_k|}(x^k)+\eta_2+\eta_1+\eta_0
=\sum_{k=1}^p\phi_{|w_k|}(x_k)+\eta_3+\eta_2+\eta_1+\eta_0\\
&=&(\sum_{k=1}^p|w_k|) \alpha+\eta_4+\eta_3+\eta_2+\eta_1+\eta_0,
\end{eqnarray*}
where
$$
\begin{cases}
|\eta_0|\leq 2|w^\prime|\|\Phi\|+2|C(\Phi)|\leq
2p(ac_1(n)+b)\|\Phi\|+2|C(\Phi)|;\\
|\eta_1|\leq 2(p-1)|C(\Phi)|;\ \ \ \ \ \ \ \ \ \ \ \
|\eta_2|\leq 2p(p_0\|\Phi\|+|C(\Phi)|);\\
|\eta_3|\leq \sum_{k=1}^{p}\|\Phi\|_{|w_k|}\leq
p\|\Phi\|_{n}^\star;\ \ \ \ \ \ \ \ \ \  |\eta_4|\leq
(\sum_{k=1}^p|w_k|)\epsilon.
\end{cases}
$$
Since $s_{p}=\sum_{k=1}^p|w_k|+p_0p$ and $|w_k|\geq C_1 n$, we have
$s_p\geq C_1np$ and
\begin{eqnarray*}
\Big |\frac{\phi_{|u|}(x)}{|u|}-\alpha\Big |
\leq\frac{|((\sum_{k=1}^p|w_k|)-|u|)\alpha|+|\eta_4|+|\eta_3|+
|\eta_2|+|\eta_1|+|\eta_0|}{|u|}.
\end{eqnarray*}
Moreover, $|u|=s_p-|w^\prime|\geq pC_1n-p(ac_1(n)+b)$  and
$\max(c_1(n)/n, \|\Phi\|_{n}^\star/n) \to 0$,  so we can choose
$N(\epsilon)$ big enough such that $|\phi_{|u|}(x)/|u|-\alpha|\leq
2\epsilon$ when $n\geq N(\epsilon)$.
 Consequently $u\in
F(\alpha,p(n+\widetilde{c}),2\epsilon).$ From this we conclude that
$$
f(\alpha,p(n+\widetilde{c}),2\epsilon)\geq
f(\alpha,n,\epsilon)^p/m^{p(ac_1(n)+b)}.
$$
 We get the desired
subadditivity by taking $\beta_n=m^{ac_1(n)+b}$.

\noindent $\bullet$\ {\it Coincidence of two limits.} Next we show
that
$$\lim_{\epsilon\to 0}\liminf_{n\to\infty}\frac{\log
f(\alpha,n,\epsilon)}{n}= \lim_{\epsilon\to
0}\limsup_{n\to\infty}\frac{\log f(\alpha,n,\epsilon)}{n}.
$$
Note that these limits exist since $f(\alpha,n,\epsilon)$ is a
non-increasing function in  the variable $\epsilon.$ Denote by
$\theta$ the left-hand side limit. Then for any $\delta>0$, there
exists $\epsilon_0>0$ such that
$$
\liminf_{n\to\infty}\log f(\alpha,n,\epsilon_0)/n<\theta+\delta.
$$
 Fix
$\delta>0$ and $\epsilon_0>0$ as above. To show the equality we only
need to show that
$$
\limsup_{n\to\infty}{\log f(\alpha,n,\epsilon_0/4)}/{n}\leq
\theta+\delta.
$$
 Take a sequence of integers $n_k\nearrow\infty$ such that
$ f(\alpha,n_k,\epsilon_0)<e^{n_k(\theta+\delta)}$ for any $ k\in
\N. $ Fix $n\geq N(\epsilon_0/4).$ For each $k,$ write
$n_k=(n+\widetilde{c})p_k-l_k$ with $0\leq l_k<n+\widetilde{c}.$ By
the subadditivity property, we have
\begin{equation}\label{I}
f(\alpha,n,\epsilon_0/4)^{p_k}\leq
\beta_n^{p_k}f(\alpha,(n+\widetilde{c})p_k,\epsilon_0/2).
\end{equation}

Next we show that there exists a positive integer sequence
$\{\gamma_k\}$ with $\gamma_k=o(p_k) $ such that
\begin{equation}\label{II}
f(\alpha,(n+\widetilde{c})p_k,\epsilon_0/2)\leq m^{\gamma_k}
f(\alpha,n_k,\epsilon_0).
\end{equation}

 Assume $w=w_1w_2$ is such that $w_1\in
{\mathcal B}_t(\Psi)$ and $w\in {\mathcal B}_{t+s}(\Psi)$ with
$1\leq s\leq (n+\widetilde{c})$, then by Lemma
\ref{Moran-covering}(4) we have
\begin{equation}\label{7}
|w_2|\leq
\frac{\Psi_{\min}-\|\Psi\|_{|w|}-(n+\widetilde{c})-2C(\Psi)}{\Psi_{\max}}.
\end{equation}
Thus $|w_2|/|w|\to 0$ when $|w|\to\infty.$  Choose $t_0$ large
enough so that when $t\geq t_0$ and $w_1\in {\mathcal B}_t(\Psi)$,
$w\in {\mathcal B}_{t+s}(\Psi)$ we have
\begin{equation}\label{6}
\frac{|w|}{|w_1|}\leq \frac{3}{2},\ \ \frac{|C(\Phi)|}{|w_1|}\leq
\frac{\epsilon_0}{16}\ \ \text{and }\ \ \frac{|w_2|}{|w_1|}\leq
\frac{\epsilon_0}{8(2\|\Phi\|+|\alpha|)}.
\end{equation}

Let $k_0$  such that $(n+\widetilde{c})p_{k_0}\geq
t_0+(n+\widetilde{c})$. Let $k\geq k_0.$  Fix $w\in
F(\alpha,(n+\widetilde{c})p_k,\epsilon_0/2)$. There exists $x\in
[w]$ such that $\big|\phi_{|w|}(x)-|w|\alpha\big|\leq
|w|\epsilon_0/2.$ Note that $(n+\widetilde c)p_k\geq n_k$. Let
$w_1\prec w$ such that $\Psi[w_1]\leq e^{-n_k} $ and
$\Psi[w_1^\ast]>e^{-n_k}$( recall that $w_1^\ast$ is obtained by
deleting the last letter of $w_1$). Thus $[w_1]\in {\mathcal
B}_{n_k}(\Psi)$. Write $w=w_1w_2.$ By $(\ref{6})$
 we have
\begin{eqnarray*}
\big|\phi_{|w_1|}(x)-|w_1|\alpha\big|&\leq&
\big|\phi_{|w_1|}(x)-\phi_{|w|}(x)\big|
+\big|\phi_{|w|}(x)-|w|\alpha\big|+|w_2||\alpha|\\
&\leq&2|w_2|\|\Phi\|+2|C(\Phi)|+\big|\phi_{|w|}(x)-|w|
\alpha\big|+|w_2||\alpha|\le |w_1|\epsilon_0,
\end{eqnarray*}
which means that $w_1\in F(\alpha,n_k,\epsilon_0)$. Write
$q_k=[C_2(\Psi)(n+\tilde c)p_k],$ then $|w|\leq q_k.$ Define
$$
\gamma_k:=\frac{\Psi_{\min}-\|\Psi\|_{q_k}
-(n+\widetilde{c})-2C(\Psi)}{\Psi_{\max}}.
$$
It is clear that $\gamma_k=o(p_k)$. Moreover by $(\ref{7})$,
$|w_2|\leq \gamma_k$. From this we conclude $(\ref{II})$.

Combine $(\ref{I})$ and $(\ref{II})$ we get
$$f(\alpha,n,\epsilon_0/4)\leq \beta_n m^{\gamma_k/p_k}
f(\alpha,n_k,\epsilon_0)^{1/p_k}\leq \beta_n m^{\gamma_k/p_k}
e^{n_k(\theta+\delta)/p_k}.$$
Letting $k\to\infty$ we get $f(\alpha,n,\epsilon_0/4)\leq \beta_n
e^{(n+\widetilde{c})(\theta+\delta)}$. Then, letting  $n\to\infty$
we have
$\limsup_{n\to\infty}{\log (f(\alpha,n,\epsilon_0/4))}/{n}\leq
\theta+\delta$.

\noindent $\bullet$ {\it Upper semi-continuity of
$\Lambda(\alpha)$.} Let $\alpha\in L_\Phi$. For any $\eta>0$ there
is $\epsilon >0$ such that
$\liminf_{n\to\infty}\log f(\alpha,n,\epsilon)/n<
\Lambda(\alpha)+\eta.$
Let $\beta\in L_\Phi$ with $|\beta-\alpha|<\epsilon/3.$ Given $w\in
F(\beta,n,\epsilon/3)$, there exists $x\in[w]$ such that
$|\phi_{|w|}(x)/|w|-\beta|\leq \epsilon/3.$ Hence
$|\phi_{|w|}(x)/|w|-\alpha|\leq
|\phi_{|w|}(x)/|w|-\beta|+|\beta-\alpha|<\epsilon,$ which means
$w\in  F(\alpha,n,\epsilon)$. This proves that $
F(\beta,n,\epsilon/3)\subset F(\alpha,n,\epsilon)$. It follows that
$f(\beta,n,\epsilon/3)\leq f(\alpha,n,\epsilon)$, therefore
$$
\Lambda(\beta)\leq
\liminf_{n\to\infty}\frac{f(\beta,n,\epsilon/3)}{n}\leq
\liminf_{n\to\infty}\frac{f(\alpha,n,\epsilon)}{n}<
\Lambda(\alpha)+\eta.
$$
This establishes the upper semi-continuity of $\Lambda$ at $\alpha.$

\noindent $\bullet$ {\it Results about $D(\Psi)$.} By essentially
repeating the same proof as above (in fact it is much easier), we
can show
$$\liminf_{n\to\infty}\frac{\log \#{\mathcal B}_n(\Psi)}{n}=
 \limsup_{n\to\infty}\frac{\log \#{\mathcal B}_n(\Psi)}{n}.$$
 We denote the  limit by $D(\Psi)$. By
 $(\ref{constant})$, for any $w\in{\mathcal B}_n(\Psi)$ we have
  $ |w|\leq Cn $, where $C=1+1/|\Psi_{\max}|.$ This yields
 $ \#{\mathcal B}_n(\Psi)\leq \#\Sigma_{A,[Cn]}$ and
 consequently $D(\Psi)\leq C\log m.$ \hfill$\Box$

Now we come to the weak concavity of the function $\Lambda$ on
$L_\Phi.$

\subsection{Proof of Proposition \ref{regularity}}

Let $A\subset\R^d$. We say that $x\in A$ is a {\it local cone
point}, or an {\it $\epsilon$-cone point } of $A$, if there exists
$\epsilon>0$ such that for any $y\in A\cap B(x,\epsilon)$ and $y\ne
x$, the interval $[x,y_\epsilon]\subset A,$ where
$y_\epsilon:=x+\epsilon(y-x)/|y-x|$.
\begin{lemma} \label{distortion-concave}
Let $A\subset \R^d$ be a  convex set and  $h: A\to \R$ be a bounded
weakly concave function. Then $h$ is lower semi-continuous at each
local cone point of $A$. Especially  $h$ is lower semi-continuous on
$\mathrm{ri}(A)$ and on any closed interval $I\subset A$. It is
lower semi-continuous on $A$ if $A\subset \R^d$ is a convex closed
polyhedron.
\end{lemma}

\begin{proof}  Let $\beta\in A$ be a $\epsilon$-cone point of $A$ for
some $\epsilon>0$. Suppose that $h$ is not lower semi-continuous at
$\beta$. Thus we can find $\eta>0$ and $\alpha_n\in A\cap
B(\beta,\epsilon)$ such that $\alpha_n\to\beta$ and $h(\alpha_n)\leq
h(\beta)-\eta.$ Define
$\alpha_n^\prime=\beta+\epsilon(\alpha_n-\beta)/|\alpha_n-\beta|$,
then $\alpha_n^\prime\in A$ since $\beta$ is a $\epsilon$-cone
point. Since $\alpha_n$ is in the open interval
$(\alpha_n^\prime,\beta)$, there exists a unique  $\lambda_n\in
(0,1)$ such that
$$
\alpha_n=\frac{\lambda_n\gamma_{1}(\alpha_n^\prime,\beta)
\alpha_n^\prime+(1-\lambda_n)\gamma_{2}(\alpha_n^\prime,\beta)\beta}
{\lambda_n\gamma_1(\alpha_n^\prime,\beta)+(1-\lambda_n)
\gamma_2(\alpha_n^\prime,\beta)},
$$
where $\gamma_1,\gamma_2$ is from the definition of weak concavity.
Since $\gamma_1,\gamma_2\in[c^{-1},c]$ and $\alpha_n\to \beta$ we
conclude that $\lambda_n\to 0.$ Since $h$ is bounded,  by
$(\ref{lower-semi-conti})$ we get $ h(\alpha_n)\geq
\lambda_nh(\alpha_n^\prime)+(1-\lambda_n)h(\beta)\to h(\beta)
 \quad ({\rm as}\ n\to\infty),
$ which is in contradiction with the choice of $\alpha_n.$ So $h$ is
lower semi-continuous at $\beta.$

Since each $x\in E$ is a local cone point of $E$ when  $E$ is ${\rm
ri}(A)$, or $E$ is a closed interval in $A$, or $E$ is A itself  and
$A$ is a convex closed polyhedron, the other results follow.
\end{proof}

Now we prove Proposition \ref{regularity}. The new point in this
proposition is the weak concavity. In fact when $d_\Psi=d_1$, as
shown in \cite{FFW}, the function $\Lambda$ is indeed concave. When
the more general metric $d_\Psi$ is considered, the length of $w\in
{\mathcal B}_n(\Psi)$ has fluctuations  (see Lemma
\ref{Moran-covering} (1)), which destroy the concavity of $\Lambda$.
However,  these  fluctuations are controllable, so that a careful
analysis yields the weak concavity of $\Lambda$.

At first we show that $\Lambda$ is bounded and
 positive. Fix $\alpha\in L_\Phi.$
  By definition $f(\alpha,n,\epsilon)\leq
   \#{\mathcal B}_n(\Psi),$ consequently $\Lambda(\alpha)\leq
  D(\Psi)$. On the other hand since $\alpha\in L_\Phi$, for any
  $\epsilon>0$, when $n$ large enough, $F(\alpha,n,\epsilon)\ne
  \emptyset.$ Consequently $\Lambda(\alpha)\geq 0.$
Thus  $\Lambda(L_\Phi)\subset [0,D(\Psi)]$.

Next we show that $\Lambda$ is weakly concave. Let $\alpha,\beta\in
L_\Phi$. For any $w_1,\cdots, w_p\in F(\alpha,n,\epsilon)$ and any
$w_{p+1}, \cdots, w_{p+q}\in F(\beta,n,\epsilon)$, let
$w=\overline{w}_1\cdots\overline{w}_{p+q}$ where
$\overline{w}_j=w_ju_j$ with $u_j\in \mathcal{W}$ such that $w$ is
admissible. By the same argument as for Proposition
\ref{dim-formular-1}, we can show that
$ \exp(-(p+q)(n+c_1(n)))\leq \Psi[w]\leq
\exp(-(p+q)(n+\widetilde{c})) $ with the same $c_1(n)$ and
$\widetilde{c}$ as in Proposition \ref{dim-formular-1}, which means
that there exists $u\prec w$ such that $u\in {\mathcal
B}_{(p+q)(n+\widetilde{c})}(\Psi)$. Write ${w}=uw^\prime$. We also
have $|w^\prime|\leq (p+q)(ac_1(n)+b)$ with the same $(a,b)$ as in
that proposition.

For any $k\in\N$ define $F_k(\alpha,n,\epsilon):=\Sigma_{A,k}\cap
F(\alpha,n,\epsilon)$. By Lemma \ref{Moran-covering} (1), we have
$\displaystyle F(\alpha,n,\epsilon) =\bigcup_{C_1n\leq k\leq
C_2n}F_k(\alpha,n,\epsilon),$
where $C_i=C_i(\Psi)$ for $i=1,2.$ Define
$f_k(\alpha,n,\epsilon)=\#F_k(\alpha,n,\epsilon).$ Choose $k_0$ such
that
$\displaystyle f_{k_0}(\alpha,n,\epsilon)=\max_{C_1n\leq k\leq
C_2n}f_k(\alpha,n,\epsilon).$
Then $f_{k_0}(\alpha,n,\epsilon)\geq
f(\alpha,n,\epsilon)/(C_2-C_1)n.$ Write $k_0=\gamma_n(\alpha)n,$
thus $\gamma_n(\alpha)\in [C_1,C_2]$. Likewise we can find
$\gamma_n(\beta)\in[C_1,C_2]$ such that
$f_{\gamma_n(\beta)n}(\beta,n,\epsilon)\geq
f(\beta,n,\epsilon)/(C_2-C_1)n.$

Fix a subsequence $n_k\uparrow \infty$ such that
$\gamma_{n_k}(\alpha)\to \gamma(\alpha)$ and $\gamma_{n_k}(\beta)\to
\gamma(\beta)$ as $k\to\infty$.  Take $w_1,\cdots,w_p\in
F_{\gamma_{n_k}(\alpha)n_k}(\alpha,n_k,\epsilon)$ and
$w_{p+1},\cdots,w_{p+q}\in
F_{\gamma_{n_k}(\beta)n_k}(\beta,n_k,\epsilon)$. Choose $x_j\in
[w_j]$
 such that
$$\begin{cases} |\phi_{|w_j|}(x_j)-|w_j|\alpha|\leq |w_j|\epsilon, & \text{ if }
1\leq j\leq p\\
|\phi_{|w_j|}(x_j)-|w_j|\beta|\leq |w_j|\epsilon, & \text{ if }
p+1\leq j\leq p+q .\end{cases}$$

  Let
$w=\overline{w}_1\cdots\overline{w}_{p+q}$ and write $w=uw^\prime$
such that $u\in {\mathcal B}_{(p+q)(n_k+\widetilde{c})}(\Psi)$. Then
we know that
$|w|=p(\gamma_{n_k}(\alpha)n_k+p_0)+q(\gamma_{n_k}(\beta)n_k+p_0)$
and $|u|=|w|-|w^\prime|.$ Now for any $x\in [w]$, define $x^1=x$ and
$x^j=T^{\sum_{l=1}^{j-1}|w_l|+p_0}x$ for $j\geq 2.$ Then we have
\begin{eqnarray*}
\phi_{|u|}(x) &=&\phi_{|w|}(x)+\eta_0
=\sum_{j=1}^{p+q}\phi_{|w_j|}(x^j)+\eta_1+\eta_0
=\sum_{j=1}^{p+q}\phi_{|w_j|}(x_j)+\eta_2+\eta_1+\eta_0\\
&=&p\gamma_{n_k}(\alpha)n_k\alpha+q\gamma_{n_k}(\beta)n_k\beta+
\eta_3+\eta_2+\eta_1+\eta_0\\
&=&p\gamma(\alpha)n_k\alpha+q\gamma(\beta)n_k\beta+\eta_4+\eta_3+\eta_2+\eta_1+\eta_0,
\end{eqnarray*}
where
$$
\begin{cases}
|\eta_0|\leq 2|w^\prime|\|\Phi\|+2|C(\Phi)|\leq
2(p+q)(ac_1(n_k)+b)\|\Phi\|+2|C(\Phi)|;\\
|\eta_1|\leq 2(p+q)(p_0\|\Phi\|+2C(\Phi));\ \ \ \ \
|\eta_2|\leq (p+q)\|\Phi\|_{n_k}^\star;\\
|\eta_3|\leq
n_k(p\gamma_{n_k}(\alpha)+q\gamma_{n_k}(\beta))\epsilon;\\
|\eta_4|\leq
pn_k|\alpha||\gamma_{n_k}(\alpha)-\gamma(\alpha)|+qn_k|\beta||\gamma_{n_k}
(\beta)-\gamma(\beta)|.
\end{cases}
$$
This yields that for $k$ large enough, $u\in
{F}((p\gamma(\alpha)\alpha+q\gamma(\beta)\beta)/(p\gamma(\alpha)+q\gamma(\beta)),
(n_k+\widetilde{c})(p+q),2\epsilon). $ Thus we conclude that
\begin{eqnarray*}
&&f\Big
(\frac{p\gamma(\alpha)\alpha+q\gamma(\beta)\beta}{p\gamma(\alpha)
+q\gamma(\beta)},(n_k+\widetilde{c})(p+q),2\epsilon\Big )\\
&\geq& [f_{\gamma_{n_k}(\alpha)n_k}(\alpha,n_k,\epsilon)]^p
[f_{\gamma_{n_k}(\beta)n_k}(\beta,n_k,\epsilon)]^qm^{-(p+q)(ac_1(n_k)+b)}\\
&\geq&
f(\alpha,n_k,\epsilon)^pf(\beta,n_k,\epsilon)^q
[(C_2-C_1)n_k]^{-p-q}m^{-(p+q)(ac_1(n_k)+b)}.
\end{eqnarray*}
Combining this with Proposition \ref{dim-formular-1} we get
$$\lambda\Lambda(\alpha)+(1-\lambda)\Lambda(\beta)\leq
\Lambda\left
(\frac{\lambda\gamma(\alpha)\alpha+(1-\lambda)\gamma(\beta)\beta}
{\lambda\gamma(\alpha)+(1-\lambda)\gamma(\beta)}\right)$$
for any $\lambda=\frac{p}{p+q}\in[0,1]\cap\Q$. Since $\Lambda$ is
upper semi-continuous, we conclude that this formula holds for any
$\lambda\in [0,1]$. Thus $\Lambda$ is weakly concave.

Assume $A\subset L_\Phi$ is a convex set, and $I\subset L_\Phi$ is a
closed interval. By Lemma \ref{distortion-concave}, $\Lambda$ is
lower semi-continuous on $\mathrm{ri}(A)$ and
  $I$.
 Combining this  with the upper semi-continuity yields the continuity
 on $\mathrm{ri}(A)$  and
 $I$. Taking $A=L_\Phi$ we get the continuity on ${\rm ri}(L_\Phi)$.

 Now assume $L_\Phi$ is a polyhedron. By Lemma
 \ref{distortion-concave}, $\Lambda$ is lower
 semi-continuous on $L_\Phi$. This,  together with the upper
 semi-continuity yields the continuity on $L_\Phi.$

 Let $I=[\alpha_1,\alpha_2]\subset L_\Phi$ and $\alpha_{\max}\in
 I$ as defined in the proposition. Assume  $\Lambda$ is
 not decreasing from $\alpha_{\max}$ to
 $\alpha_1$. Since $\Lambda$ is continuous on $I$,
 we can find $\beta_1,\beta_2, \beta_3\in [\alpha_1,\alpha_{\max}]$ such that
 $\beta_2\in [\beta_1,\beta_3]$ and $
 \Lambda(\beta_1)=\Lambda(\beta_3)>\Lambda(\beta_2),
 $
which is in contradiction with the fact that $\Lambda$ is
quasi-concave, since it is weakly concave. Thus
 $\Lambda$ is decreasing from $\alpha_{\max}$ to
 $\alpha_1$. The same argument  shows that $\Lambda$ is
 decreasing from $\alpha_{\max}$ to
 $\alpha_2$.
\hfill$\Box$

\section{Proof of  Theorem~ \ref{main-one-sided}(1)}\label{proof of theo}

By Proposition \ref{bsic-aa}, we have $E_\Phi(\alpha)\ne\emptyset$
if and only if $\alpha\in L_\Phi.$ For the next statement, our plan
is the following: we show that $\D(\alpha)\leq \Lambda(\alpha)\leq
{\mathcal E}(\alpha)\leq \D(\alpha)$. We divide this into three
steps corresponding to the next Sections~\ref{upperbd}, \ref{5.2}
and \ref{5.3}.

\subsection{
$\D(\alpha)\leq \Lambda(\alpha)$}\label{upperbd} We prove a slightly
more general result for the upper bound.  Given
$\Phi\in\C_{aa}(\Sigma_A,T,d)$ and $\Omega\subset L_\Phi$, define
$\displaystyle E_\Phi(\Omega):=\bigcup_{\alpha\in
\Omega}E_\Phi(\alpha). $

\begin{proposition}\label{upper-bound} For any compact set $\Omega\subset L_\Phi$
 we have $\dim_P E_\Phi(\Omega)\leq \sup\{\Lambda(\alpha):
\alpha\in\Omega\}$. In particular, if $\alpha\in L_\Phi$ we have
$\D(\alpha)\leq \dim_P E_\Phi(\alpha)\leq \Lambda(\alpha)$.
\end{proposition}

\begin{proof}
 Let
$\displaystyle
\Lambda(\alpha,\epsilon):=\limsup_{n\to\infty}\frac{\log
f(\alpha,n,\epsilon)}{n}, $ then $\Lambda(\alpha,\epsilon)\searrow
\Lambda(\alpha) $ when $\epsilon\searrow 0.$ Fix $\eta>0$. For each
$\alpha\in \Omega,$ there exists $\epsilon_\alpha>0$ such that for
any $0<\epsilon\leq \epsilon_\alpha$ we have $
\Lambda(\alpha,\epsilon)<\Lambda(\alpha)+\eta. $ Since
$\{B(\alpha,\epsilon_\alpha): \alpha\in \Omega\}$ is an open
covering of $\Omega$, we can find a finite covering
 $\{B(\alpha_1,\epsilon_1),\cdots, B(\alpha_s,\epsilon_s)\},$ where
$\epsilon_j=\epsilon_{{\alpha_j}}$. For each $n\in \N$ define
$$
H(n,\eta):= \bigcup_{j=1}^{s}\bigcup_{w\in
F(\alpha_j,n,\epsilon_j)}[w]\ \ \text{ and }\ \
G(k,\eta):=\bigcap_{n\geq k}H(n,\eta)$$

It is standard to prove that  $E_\Phi(\Omega)\subset
\bigcup_{k\in\N}G(k,\eta).$ Consequently

\begin{equation}\label{sup}
\dim_P E_\Phi(\Omega)\leq \sup_{k\in\N} \dim_P G(k,\eta).
\end{equation}
By definition, the set $ G(k,\eta)$ is covered by $\{[w]: w\in
F(\alpha_j,n,\epsilon_j); j=1,\cdots, s\}$ for any $n\geq k$. Since
each element in $\{[w]: w\in F(\alpha_j,n,\epsilon_j)\}$ is a ball
with radius $e^{-n}$, we get
\begin{eqnarray*}
&& \dim_{P} G(k,\eta) \leq\overline{\dim}_B G(k,\eta) \leq
\limsup_{n\to\infty}\frac{\log\sum_{j=1}^{s}
f(\alpha_j,n,\epsilon_j)}{n}\\
& \leq& \sup_{j=1,\cdots,s}\limsup_{n\to\infty}\frac{\log
f(\alpha_j,n,\epsilon_j)}{n}
 =
\sup_{j=1,\cdots,s}\Lambda(\alpha_j,\epsilon_j)
\leq\sup\{\Lambda(\alpha): \alpha\in \Omega\}+\eta.
\end{eqnarray*}

Combining this with $(\ref{sup})$ we get $\dim_P E_\Phi(\Omega)\leq
\sup\{\Lambda(\alpha): \alpha\in\Omega\}+\eta.$
\end{proof}

\subsection{
$\Lambda(\alpha)\leq {\mathcal E}(\alpha)$}\label{5.2} Our approach
is inspired by that of \cite{FFW}, which deals with the case that
$d_\Psi=d_1$ and $\Phi=(S_n\phi)_{n\geq 1}$ is additive, where
$\phi:\Sigma_A\to\R^d$ is continuous.

 To show this
inequality we need to approximate the almost additive potentials
$\Phi$ and $\Psi$ by two sequences of H\"{o}lder potentials. We
describe this procedure as follows.

 Given $\Phi\in\C_{aa}(\Sigma_A,T,d)$, for each $k\in\N$ we define
${\Phi}^{(k)}\in \C_{aa}(\Sigma_A,T,d)$  as follows. For each
$w\in\Sigma_{A,k}$ choose $x_w\in [w]$. For any $x\in[w]$ define $
\widetilde{\phi}_k(x):={\phi_k(x_w)}/{k}. $ Then
$\widetilde{\phi}_k$ depends only on the first $k$ coordinates of
$x\in\Sigma_A$ and  is H\"{o}lder continuous. Define
\begin{equation}\label{app-potential}
{\Phi}^{(k)}=(S_n\widetilde{\phi}_k)_{n=1}^{\infty}.
\end{equation}
  Thus ${\Phi}^{(k)}$ is additive and  H\"{o}lder continuous.

\begin{lemma}\label{control}
  We have $\Phi_{\min}\leq
\Phi_{\min}^{(k)}\leq\Phi_{\max}^{(k)}\leq \Phi_{\max}$. Moreover
$$
\|\phi_n-S_n\widetilde{\phi}_k\|\leq
d\Big(\frac{n}{k}|C(\Phi)|+5k\|\Phi\|+\frac{\|\Phi\|_k}{k}n\Big).
$$ Consequently $\|\Phi-\Phi^{(k)}\|_{\mbox{\tiny \rm lim}}\to 0$
when $k\to\infty.$
\end{lemma}
This lemma will be proved at the end of this section.

\noindent{\bf Proof of $\Lambda(\alpha)\leq {\mathcal E}(\alpha)$.}\
Given ${\Phi}\in \C_{aa}(\Sigma_A,T,d)$ and $\Psi\in
\C_{aa}^-(\Sigma_A,T)$.

\noindent {\bf Claim:} Given $\epsilon>0$ we have
\begin{equation} \label{claim}
\limsup_{n\to\infty}\frac{\log
f(\alpha,n,\epsilon,\Phi,\Psi)}{n}\leq\sup_{|\Phi_\ast(\nu)-\alpha|\leq
5\epsilon}\frac{h_\nu}{-\Psi_\ast(\nu)+O(\epsilon)}.
\end{equation}

Let us at first assume the claim holds and finish the proof. Notice
that the set of invariant measures $\nu$ such that
$|\Phi_\ast(\nu)-\alpha|\leq 5\epsilon$ is compact, so by using the
upper semi-continuity of $h_\nu$  and letting $\epsilon$ tend to 0
we can find an invariant measure $\nu_0$ such that
$\Phi_\ast(\nu_0)=\alpha$ and
$$
\Lambda(\alpha)=\lim_{\epsilon\to 0} \limsup_{n\to\infty} \frac{\log
f(\alpha,n,\epsilon,\Phi,\Psi)}{n} \leq
\frac{h_{\nu_0}}{-\Psi_\ast(\nu_0)}\leq {\mathcal E}(\alpha).
$$

Next we show that the claim holds. In the following $C=C(\Psi),
C_i=C_i(\Psi)$ for $i=1,2.$ Fix $\epsilon>0$. For any $k\in\N$
define $\Phi^{(k)} $ and $\Psi^{(k)}$ according to
$(\ref{app-potential})$.
 By Lemma \ref{control}, we can find $k\in\N$
such that
$$
\|\Phi-\Phi^{(k)}\|_{\mbox{\tiny \rm lim}},\ \
\|\Psi-\Psi^{(k)}\|_{\mbox{\tiny \rm lim}}<\epsilon.
$$
Fix this $k$, then there exists $N_1\in\N$ such that when $n\geq
N_1$
 $$
 \|\phi_n-S_n\widetilde{\phi}_k\|_\infty\leq n\epsilon\ \
 \text{ and } \ \ \|\psi_n-S_n\widetilde{\psi}_k\|_\infty\leq n\epsilon.
 $$

 For any $w\in {\mathcal B}_n(\Psi)$, we have $C_1 n\leq |w|\leq C_2n,$
 thus for $n\geq N_1/C_1$, we have
$F(\alpha,n,\epsilon,\Phi,\Psi)\subset
F(\alpha,n,2\epsilon,\Phi^{(k)},\Psi),$ consequently
\begin{equation}\label{card-3}
f(\alpha,n,\epsilon,\Phi,\Psi)\leq
f(\alpha,n,2\epsilon,\Phi^{(k)},\Psi).
\end{equation}

Following \cite{FFW}, we introduce a way to classify the words in
$F(\alpha,n,2\epsilon,\Phi^{(k)},\Psi),$ by which we can estimate
the cardinality of it effectively.

 For any word $w\in \Sigma_{A,\ast}$ such that $|w|\geq
k$, we define the counting  function $\theta_w:\Sigma_{A,k}\to \N$
as $\theta_w(u)=\#\{j: w_j\cdots w_{j+k-1}=u\}$, which counts the
numbers of times the word $u$ appears in $w.$ It is clear that $
h_{\theta_w}:=\sum_u\theta_w(u)=|w|-k+1. $ We call it the {\it
height} of $\theta_w.$

Let
${\mathcal P}_k^{(n)}=\{\theta_w: w\in
F(\alpha,n,2\epsilon,\Phi^{(k)},\Psi)\}.$
Then $\#{\mathcal P}_k^{(n)}\leq (C_2n)^{m^k}$. For each $\theta\in
{\mathcal P}_k^{(n)}$, let
$$
{\mathcal T}(\theta)=\{w: w\in
F(\alpha,n,2\epsilon,\Phi^{(k)},\Psi),\ \theta_{w}=\theta\}.
$$
Write $\Gamma(\theta):=\#{\mathcal T}(\theta).$  Then we have
$$f(\alpha,n,2\epsilon,\Phi^{(k)},\Psi)=
 \sum_{\theta\in{\mathcal P}_k^{(n)}}\Gamma(\theta)
\leq (C_2n)^{m^k}\max_{\theta\in{\mathcal
P}_k^{(n)}}\Gamma(\theta).$$
Consequently
\begin{equation}\label{card-4}
 \frac{\log f(\alpha,n,2\epsilon,\Phi^{(k)},\Psi)}{n}\leq
\max_{\theta\in{\mathcal P}_k^{(n)}}\frac{\log
\Gamma(\theta)}{n}+m^kO(\frac{\log n}{n}).
\end{equation}

In the following  we  estimate $\log \Gamma(\theta)/n$ for each
$\theta\in {\mathcal P}_k^{(n)}.$ Since it is hard to estimate it
directly, we turn to the estimations of $\log
\Gamma(\theta)/h_\theta$ and $n/h_\theta.$

Following ~\cite{FFW} we define $\triangle_k^+$, the set of all
positive functions $p$ on $\Sigma_{A,k}$ satisfying the following
two relations:
$$
\sum_{w\in\Sigma_{A,k}}p(w)=1;\ \  \sum_{w}p(ww_1w_2\cdots
w_{k-1})=\sum_{w}p(w_1w_2\cdots w_{k-1}w).
$$

It is  known (see \cite{FFW}) that for any $\eta>0$, there is a
positive integer
 $N=N(\eta)$ such that for any $w\in \Sigma_{A,l+k-1}$ with
$l>N$, there exists a probability vector $p\in \triangle_k^+$ such
that
$$\Big|\frac{\theta_w(u)}{l}-p(u)\Big |<\eta, \ \ p(u)>\frac{\eta}{m^{k+1}}.$$
We discard the trivial case where $\Phi\equiv 0$ and fix  $\eta>0$
such that $\eta<\epsilon/(m^k\|\Phi\|).$

Now take any $\displaystyle n\geq \max\{\frac{N_1}{C_1},
\frac{4k\|\Phi\|}{C_1\epsilon}\}$ and fix a $\theta\in {\mathcal
P}_k^{(n)}$.
 Take
 $$
 w\in
F(\alpha,n,2\epsilon,\Phi^{(k)},\Psi)
$$ such that
$\theta=\theta_{w}$, then $|w|=h_\theta+k-1.$ Fix a
$p\in\triangle_k^+$ as described above. Consider the Markov measure
$\nu_p$ corresponding to $p$ (see ~\cite{FFW} for the definition and
related properties). For any word $v\in{\mathcal T}(\theta)$ we have
\begin{eqnarray*}
\nu_p([v])=\frac{p(v|_k)}{t(v|_k)}\prod_{|u|=k} t(u)^{\theta(u)}
\geq \frac{\eta}{m^{k+1}}\prod_{|u|=k} t(u)^{\theta(u)}:=\rho,
\end{eqnarray*}
where
$$
\displaystyle t(a_1\cdots a_k)=\frac{p(a_1\cdots a_k)}
{\sum_{\epsilon}p(a_1\cdots a_{k-1}\epsilon)}.
$$ Thus
$\rho\cdot\Gamma(\theta)\leq \nu_p(\bigcup_{v\in{\mathcal T}
 (\theta)}[v])\leq 1$ and consequently
$$\Gamma(\theta)\leq \frac{1}{\rho}=
\frac{m^{k+1}}{\eta}\prod_{|u|=k}t(u)^{-\theta(u)}.$$
Since $C_1n\leq |w|=h_\theta+k-1\leq C_2n$, we have $h_\theta\sim n$
when $n\to\infty.$ Notice that $\eta/m^{k+1}\leq t(u)\leq 1$, thus
\begin{eqnarray}
\nonumber\frac{\log\Gamma(\theta)}{h_\theta}&\leq& O(\frac{k}{n})+
O(\frac{|\log\eta|}{n})-\sum_{|u|=k}\frac{\theta(u)}{h_\theta}\log
t(u)\\
\nonumber\label{card-1}&\leq& o(1)+
O(\frac{|\log\eta|}{n})-\sum_{|u|=k}p(u)\log
t(u)+m^k\eta(|\log \eta|+(k+1)\log m)\\
&=& h_{\nu_p}+ o(1)+ O(\frac{|\log\eta|}{n})+O(\eta|\log\eta|).
\end{eqnarray}

Now we estimate $n/h_\theta.$ Let $x_0\in[w]$. By $(\ref{444})$ we
have
$$
 -n-C(\Psi)-2\|\Psi\|_{|w|}+\Psi_{\min}\leq
\psi_{|w|}(x_0)\leq\sup_{x\in[w]}\psi_{|w|}(x)\leq -n,
$$
Since $n\geq N_1/C_1,$ we have $|w|\geq C_1n\geq N_1$. Thus
$\|\psi_{|w|}-S_{|w|}\widetilde{\psi}_k\|_\infty\leq |w|\epsilon$.
We get
$$
-n-C(\Psi)-2\|\Psi\|_{n}^\star+\Psi_{\min}-C_2n\epsilon\leq
S_{|w|}\widetilde{\psi}_k(x_0)\leq -n+C_2n\epsilon.
$$
Notice that $\widetilde \psi_k$ is negative and $|w|=h_\theta+k-1$,
thus
\begin{equation}\label{comparible}
 -n-C(\Psi)-2\|\Psi\|_{n}^\star+\Psi_{\min}-C_2n\epsilon\leq
S_{h_\theta}\widetilde{\psi}_k(x_0)\leq -n+C_2n\epsilon+k\|\Psi\|.
\end{equation}
On the other hand
\begin{eqnarray*}
&&\frac{S_{h_\theta}\widetilde{\psi}_k(x_0)}{h_\theta}=
\sum_{|u|=k}\frac{\theta(u)}{h_\theta}\widetilde{\psi}_k(x_u)
=\sum_{|u|=k}p(u)\widetilde{\psi}_k(x_u)+m^kO(\eta)\\
&=&\int\widetilde{\psi}_k d\nu_p+O(\eta)
={\Psi}^{(k)}_\ast(\nu_p)+O(\eta)
=\Psi_\ast(\nu_p)+O(\epsilon)+O(\eta).
\end{eqnarray*}
Combining this with $(\ref{comparible})$ and the fact that
$\|\Psi\|_n^\star/n=o(1)$ we get
\begin{equation}\label{card-2}
\frac{n}{h_\theta}=-\Psi_\ast(\nu_p)+O(\epsilon)+O(\eta)+o(1).
\end{equation}
 Combine $(\ref{card-1})$ and $(\ref{card-2})$ we get
\begin{eqnarray}\label{card-5}
\frac{\log\Gamma(\theta)}{n} &\leq&\frac{h_{\nu_p}+ o(1)+
O(\frac{|\log\eta|}{n})+O(\eta|\log\eta|)}
{-\Psi_\ast(\nu_p)+O(\epsilon)+O(\eta)+o(1)}.
\end{eqnarray}

Next we  show that $|\Phi_\ast(\nu_p)-\alpha|\leq 5\epsilon.$
 Since $w\in F(\alpha,n,2\epsilon,\Phi^{(k)},\Psi)$, there
 exists $y_0\in[w]$ such that
 $|S_{|w|}\widetilde{\phi}_k(y_0)/|w|-\alpha|\leq 2\epsilon.$ Note that
 $|w|=h_{\theta_w}+k-1,$ we have
 \begin{eqnarray*}
 &&|\Phi_\ast(\nu_p)-\alpha|
 \leq  |{\Phi}^{(k)}_\ast(\nu_p)-\alpha|+
 |{\Phi}_\ast(\nu_p)-{\Phi}^{(k)}_\ast(\nu_p)|\\
 &\leq&  |\int\widetilde{\phi}_k
 d\nu_p-\alpha|+\epsilon
 =  |\sum_{|u|=k}p(u)\widetilde{\phi}_k(x_u)-\alpha|+\epsilon\\
 &\leq&  |\sum_{|u|=k}\frac{\theta_{w}(u)}{h_{\theta_w}}
 \widetilde{\phi}_k(x_u)-\alpha|+m^k\eta\|\widetilde{\phi}_k\|+\epsilon \ \
 \\
 &\leq&  |\frac{S_{h_{\theta_w}}{\widetilde\phi}_k(x)}{h_{\theta_w}}
 -\alpha|+m^k\eta\|\Phi\|+\epsilon \
  \ (\text{ for any } x\in[w])\\
  &\leq&  |\frac{S_{|w|}{\widetilde\phi}_k(x)}{|w|}
 -\alpha|+m^k\eta\|\Phi\|+\epsilon+\frac{2k\|\Phi\|}{|w|} \
  \ (\text{ for any } x\in[w])\\
 &\leq&|\frac{S_{|w|}{\widetilde\phi}_k(x)}{|w|}
 -\frac{S_{|w|}{\widetilde\phi}_k(y_0)}{|w|}|+
 |\frac{S_{|w|}{\widetilde\phi}_k(y_0)}{|w|}-\alpha|
 +m^k\eta\|\Phi\|+\epsilon +\frac{2k\|\Phi\|}{|w|}\\
 &\leq&\frac{2k\|\Phi\|}{|w|}+
 2\epsilon
 +m^k\eta\|\Phi\|+\epsilon +\frac{2k\|\Phi\|}{|w|}\leq
 \frac{4k\|\Phi\|}{C_1n} +m^k\eta\|\Phi\|+3\epsilon.
\end{eqnarray*}
By our choice of $\eta$ we have $m^k\eta\|\Phi\|<\epsilon$.
Moreover, since  $n\geq 4k\|\Phi\|/(C_1\epsilon)$ we have
$4k\|\Phi\|/(C_1n)\leq \epsilon.$ Thus
$|\Phi_\ast(\nu_p)-\alpha|\leq 5\epsilon.$

 Again by the compactness of the set
  $\{\nu: |\Phi_\ast(\nu)-\alpha|\leq 5\epsilon \}$
   and the  upper semi-continuity of
 $h_\nu$,
  Combining  $(\ref{card-3}), (\ref{card-4})$ and $(\ref{card-5})$ we
conclude that
$$
\limsup_{n\to\infty}\frac{\log f(\alpha,n,\epsilon,\Phi,\Psi)}{n}
\leq \sup_{|\Phi_\ast(\nu)-\alpha|\leq 5\epsilon}\frac{h_\nu+
O(\eta|\log\eta|)} {-\Psi_\ast(\nu)+O(\epsilon)+O(\eta)}.
$$
Let $\eta\to 0$ we get $(\ref{claim}).$\hfill $\Box$

\medskip

\noindent {\bf Proof of Lemma~\ref{control}.} At first we assume
$\Phi\in\C_{aa}(\Sigma_A,T)$. By $(\ref{max-min})$ we get
$\Phi_{\min}\leq\widetilde{\phi}_k\leq \Phi_{\max}.$ Since
$\Phi^{(k)}$ is additive, we have $
\Phi_{\min}\leq\widetilde{\phi}_{k\min}=
\Phi_{\min}^{(k)}\leq\Phi_{\max}^{(k)}=\widetilde{\phi}_{k\max}\leq
\Phi_{\max}. $

For $n\in \N$, write $n=pk+s$ with $0\leq s<k.$ Write $C=C(\Phi)$,
by using the almost additivity of $\Phi$ for each $0\leq j\leq k-1$
we have
\begin{eqnarray*}
\phi_n(x)&\leq&\phi_j(x)+\sum_{l=0}^{p-2}\phi_k(T^{j+lk}x)+
\phi_{n-(j+(p-1)k)}(T^{j+(p-1)k}x)+pC\\
&\leq& pC+3k\|\Phi\|+\sum_{l=0}^{p-2}\phi_k(T^{j+lk}x),
\end{eqnarray*}
hence
\begin{eqnarray*}
\phi_n(x) &\leq&
pC+3k\|\Phi\|+\sum_{j=0}^{k-1}\sum_{l=0}^{p-2}\phi_k(T^{j+lk}x)/k\\
&\leq & pC+5k\|\Phi\|+\sum_{j=0}^{n-1}\phi_k(T^jx)/k\le
S_n\widetilde{\phi}_k(x)+pC+5k\|\Phi\|+\frac{\|\Phi\|_k}{k}n.
\end{eqnarray*}
Similarly we have $ \phi_n(x) \geq
S_n\widetilde{\phi}_k(x)-pC-5k\|\Phi\|-\frac{\|\Phi\|_k}{k}n,$ hence
$$
\|\phi_n-S_n\widetilde{\phi}_k\|_\infty\leq
pC+5k\|\Phi\|+\frac{\|\Phi\|_k}{k}n \  \text{ and }\
\|\Phi-\Phi^{(k)}\|_{\mbox{\tiny \rm lim}}\leq
\frac{C}{k}+\frac{\|\Phi\|_k}{k}\to 0\ (k\to\infty).
$$

If $\Phi=(\Phi^1,\cdots,\Phi^d)\in\C_{aa}(\Sigma_A,T,d)$,
 applying the result just proven to each component of $\Phi$
  we get the result. \hfill $\Box$

 \subsection{ ${\mathcal E}(\alpha)\leq
 \D(\alpha)$}\label{5.3}

 It is contained in Proposition \ref{lower-unify} (see Section~\ref{proofmainth}).

\section{Proof of  Theorem~ \ref{main-one-sided}(2)}

 We need to describe the $\Psi$- and $\Phi$-
dependence of the function $\Lambda=\Lambda_\Phi^\Psi.$  Recall that
 $$\Lambda_{\Phi}^{\Psi}(\alpha,\epsilon)=
 \limsup_{n\to\infty}\frac{\log f(\alpha,n,\epsilon,\Phi,\Psi)}{n}$$
 and we know that
 $\Lambda_{\Phi}^{\Psi}(\alpha,\epsilon)\searrow\Lambda_{\Phi}^{\Psi}(\alpha)$
 as $\epsilon \searrow 0$.

 \begin{lemma}\label{Psi-dependence}
  \begin{enumerate}
  \item Assume $\Psi,\Upsilon\in \C_{aa}^{-}(\Sigma_A,T)$, then we have
 \begin{equation}\label{approximation-1}
 |D(\Psi)-D(\Upsilon)|\leq 3\log m\cdot
 \Big (1+\frac{1}{|\Psi_{\max}|}\Big )\Big
 (1+\frac{1}{|\Upsilon_{\max}|}\Big )\|\Psi-\Upsilon\|_{\mbox{\tiny \rm lim}}.
 \end{equation}

\item Assume
$\delta_0:=\|\Psi-\Upsilon\|_{\mbox{\tiny \rm lim}}\leq
1/(4C_2(\Psi))$. Let  $\Phi,\Theta\in \C_{aa}(\Sigma_A,T,d)$. Fix
$\eta>0$ and let $\beta\in L_\Theta.$ Then for any $\alpha\in
B(\beta,\eta)\cap L_\Phi$ we have
 \begin{equation}\label{approximation-3}
 \Lambda_{\Phi}^{\Psi}(\alpha)\leq \frac{2C_2(\Psi)\log
m}{|\Upsilon_{\max}|}\delta_0+(1-C_2(\Psi)\delta_0)\Lambda_{\Theta}^{\Upsilon}(\beta,
a_0+\kappa\delta_0+2\eta),
 \end{equation}
 where
$a_0=\|\Phi-\Theta\|_{\mbox{\tiny \rm lim}}$ and
$\kappa=\kappa(\Psi,\Upsilon,\Phi)={18\|\Phi\|C_2(\Psi)}
|\Upsilon_{\min}|/|{\Upsilon_{\max}}|.$
\end{enumerate}
\end{lemma}

 \medskip

Let us prove Theorem \ref{main-one-sided} (2). Suppose first that
$\Psi$ is H\"{o}lder continuous. Let $t$ be the solution of the
equation $P(t\Psi)=0$ and  $\mu$ be the unique equilibrium state of
$t\Psi$, where $P(t\Psi)$ is the topological pressure of $t\Psi$.
The measure $\mu$ is ergodic (\cite{Bow75}), and
$\dim_H\Sigma_A=\dim_H\mu $ (\cite{Bow79}). Moreover $t$ is also the
box dimension of $(\Sigma_A,d_\Psi)$. Consequently, $t=D(\Psi)$ by
\eqref{full-dim}. Let $\alpha=\Phi_\ast(\mu)$. By the sub-additive
ergodic theorem we have $\mu(E_\Phi(\alpha))=1$, consequently
$\D(\alpha)=D(\Psi).$ Thus, when $\Psi$ is a H\"{o}lder potential
the result holds.

Next we  assume $\Psi\in\C_{aa}^{-}(\Sigma_A,T)$. Define
$\Psi^{(n)}$ according to $(\ref{app-potential})$, then we have
$$
\lim_{n\to\infty}\|\Psi-\Psi^{(n)}\|_{\mbox{\tiny \rm lim}}= 0\ \
\text{ and }\ \
|\Psi_{\max}|\leq|\Psi_{\max}^{(n)}|\leq|\Psi_{\min}|.
$$ and  by
$(\ref{approximation-1})$ we have $\lim_{n\to\infty} D(\Psi^{(n)})=
D(\Psi)$. Let $\mu_n$ be the unique equilibrium state of
$D(\Psi^{(n)})\cdot\Psi^{(n)}$ and define
$\alpha_n=\Phi_\ast(\mu_n)$. Then $\alpha_n\in L_\Phi$ and
$\Lambda_{\Phi}^{\Psi^{(n)}}(\alpha_n)=D(\Psi^{(n)}).$ Let $\alpha$
be a limit point of the sequence $\{\alpha_n: n\in \N\}$. Without
loss of generality we assume $\alpha=\lim_{n\to\infty}\alpha_{n}$.
By $(\ref{approximation-3})$ we have
\begin{equation}\label{555}
\Lambda_{\Phi}^{\Psi^{(n)}}(\alpha_n)\leq
 \frac{2C_2(\Psi^{(n)})\log
m}{|\Psi_{\max}|}\delta_n+(1-C_2(\Psi^{(n)})\delta_n)
 \Lambda_{\Phi}^{\Psi}(\alpha,
 \kappa_n\delta_n+2\eta_n),
 \end{equation}
where $\delta_n:=\|\Psi-\Psi^{(n)}\|_{\mbox{\tiny \rm lim}}$ and
$$
 C_2(\Psi^{(n)})=1+\frac{1}{|\Psi_{\max}^{(n)}|},\
\kappa_n=\frac{18\|\Phi\|C_2(\Psi^{n})|\Psi_{\min}|}{
{|\Psi_{\max}|}},\ \eta_n=|\alpha-\alpha_n|.
$$
By Lemma \ref{control} we have  $C_2(\Psi^{(n)})\leq
1+1/|\Psi_{\max}|, $ thus we can rewrite $(\ref{555})$ as
$$
D(\Psi^{(n)})\leq d_1\delta_n+\Lambda_{\Phi}^{\Psi}(\alpha,
d_2\delta_n+2\eta_n).
$$
Letting $n$ tend to $\infty$ we get $D(\Psi)\leq
\Lambda_{\Phi}^{\Psi}(\alpha).$ By the definition of box dimension
we have $\dim_B\Sigma_A\leq D(\Psi)$. Thus we have
$$
D(\Psi)\leq\Lambda_{\Phi}^{\Psi}(\alpha)=\dim_H E_\Phi(\alpha)\leq
\dim_H \Sigma_A\leq \dim_B\Sigma_A\leq D(\Psi),
$$
and we get the equality. \hfill $\Box$

\noindent {\bf Proof of Lemma~\ref{Psi-dependence}.} (1) \ Write
$\Psi=(\psi_n)_{n=1}^{\infty}$ and
$\Upsilon=(\upsilon_n)_{n=1}^{\infty}.$
  By the definition of $\|\cdot\|_{\mbox{\tiny \rm lim}}$,
  for any $\delta>\|\Psi-\Upsilon\|_{\mbox{\tiny \rm lim}},$
  there exist $N\in\N,$ such that for any $n\geq N$ we have
 $$
 \psi_n(x)-n\delta\leq \upsilon_n(x)\leq \psi_n(x)+n\delta,
 $$
 consequently for any $w\in \Sigma_{A,\ast}$ with $|w|$ large enough we have
 $$\Psi[w]e^{-|w|\delta}\leq\Upsilon[w]\leq\Psi[w] e^{|w|\delta}.$$
 Given $w\in {\mathcal B}_n(\Psi)$, by $(\ref{constant})$ and $(\ref{444})$ we have
 $$
 e^{\Psi_{\min}-C(\Psi)-\|\Psi\|_{n}^\star-n(1+C_2(\Psi)\delta)}
 \leq\Upsilon[w]\leq e^{-n(1-C_2(\Psi)\delta)}.
 $$
 This implies that there exists $u\prec w$ such
that $u\in {\mathcal B}_{[n(1-C_2(\Psi)\delta)]}(\Upsilon)$, hence
we have
\begin{equation}\label{upper}
\#{\mathcal B}_{[n(1-C_2(\Psi)\delta)]}(\Upsilon)\leq \#{\mathcal
B}_{n}(\Psi).
\end{equation}

Let  $c_1(n)=-\Psi_{\min}+C(\Psi)+\|\Psi\|_{n}^\star$. We have
$c_1(n)>0$ and $c_1(n)=o(n)$.  Write $w=uw^\prime$. The same proof
as that of
 the claim in Proposition \ref{dim-formular-1} yields
 $|w^\prime|\leq (c_1(n)+2nC_2(\Psi)\delta+C(\Upsilon))/
 |\Upsilon_{\max}|.$ Thus we can conclude that
\begin{equation}\label{lower}
\#{\mathcal B}_{[n(1-C_2(\Psi)\delta)]}(\Upsilon)\geq \#{\mathcal
B}_{n}(\Psi)m^{-\big(c_1(n)+2nC_2(\Psi)\delta+C(\Upsilon)\big)/
 |\Upsilon_{\max}|}.
 \end{equation}
 Combining $(\ref{upper})$, $(\ref{lower})$ and $(\ref{full-dim})$ we get
 $$\big(1-C_2(\Psi)\delta\big)D(\Upsilon)\leq D(\Psi)\leq
 \big(1-C_2(\Psi)\delta\big)D(\Upsilon)+2C_2(\Psi)\delta\log m/|\Upsilon_{\max}|.$$
By using  $(\ref{ful-dim-1})$  we get $|D(\Psi)-D(\Upsilon)|\leq
a(m,\Psi,\Upsilon)\delta,$ where
$$a(m,\Psi,\Upsilon)=3C_2(\Psi)C_2(\Upsilon)\log
m=3\Big (1+\frac{1}{|\Psi_{\max}|}\Big )\Big
(1+\frac{1}{|\Upsilon_{\max}|}\Big )\log m.$$ Since
$\delta>\|\Psi-\Upsilon\|_{\mbox{\tiny \rm lim}}$ is arbitrary, we
get $|D(\Psi)-D(\Upsilon)|\leq
a(m,\Psi,\Upsilon)\|\Psi-\Upsilon\|_{\mbox{\tiny \rm lim}}.$

\medskip

(2) Now let  $\Phi,\Theta\in\C_{aa}(\Sigma_A,T,d)$. Fix
$0<\epsilon<\|\Phi\|$, $\beta\in L_\Theta$, $\alpha\in
B(\beta,\eta)\cap L_\Phi$, and
$\delta>\|\Psi-\Upsilon\|_{\mbox{\tiny \rm lim}}$.

Let $w\in F(\alpha,n,\epsilon,\Phi,\Psi)$. There exists $x\in [w]$
such that
 $\big|\phi_{|w|}(x)-|w|\alpha\big|\leq |w|\epsilon.$
We have seen in proving (1) that  $w=uw^\prime$ with $u\in {\mathcal
B}_{[n(1-C_2(\Psi)\delta)]}(\Upsilon)$ and
$$
|w^\prime|\leq (c_1(n)+2nC_2(\Psi)\delta+C(\Upsilon))/
 |\Upsilon_{\max}|.
 $$
  Notice that $\mathrm{diam}(L_\Phi)\leq \|\Phi\|,$
 thus $|\alpha|\leq \|\Phi\|.$ So we have
 \begin{eqnarray*}
 &&\big|\phi_{|u|}(x)-|u|\alpha\big|
 \leq
 \big|\phi_{|u|}(x)-\phi_{|w|}(x)\big|+\big|\phi_{|w|}(x)-|w|\alpha\big|+
 |w^\prime||\alpha|\\
 &\leq&
 2|w^\prime|\cdot\|\Phi\|+2C(\Phi)+|w|\epsilon+|w^\prime||\alpha|\leq
 4|w^\prime|\cdot\|\Phi\|+2|C(\Phi)|+|u|\epsilon\\
 &=&|u|\Big (\epsilon+\frac{4|w^\prime|\cdot\|\Phi\|+2|C(\Phi)|}{|u|}\Big ).
 \end{eqnarray*}
Since $0<c_1(n)=o(n)$, for large $n$ we have
$$
4|w^\prime|\cdot\|\Phi\|+2|C(\Phi)|\leq
\frac{9n\|\Phi\|C_2(\Psi)}{|\Upsilon_{\max}|}\delta.
$$
Moreover if $\delta< 1/(2C_2(\Psi))$, then
 $$
 |u|\geq
C_1(\Upsilon)n(1-C_2(\Psi)\delta)\geq \frac{n}{2|\Upsilon_{\min}|}.
$$
Thus we get
$$
\frac{4|w^\prime|\cdot\|\Phi\|+2|C(\Phi)|}{|u|} \leq
\frac{18\|\Phi\|C_2(\Psi)|\Upsilon_{\min}|}{|\Upsilon_{\max}|}\delta
=:\kappa(\Psi,\Upsilon,\Phi)\delta=\kappa\delta.
$$
Fix any $a>\|\Theta-\Phi\|_{\mbox{\tiny \rm lim}}$. For $n$ large
enough we have
\begin{eqnarray*}
\big|\theta_{|u|}(x)-|u|\beta\big|&\leq &
\big|\theta_{|u|}(x)-\phi_{|u|}(x)\big|+
\big|\phi_{|u|}(x)-|u|\alpha\big|+|u||\alpha-\beta|\\
&\leq&a|u|+(\epsilon+\kappa\delta)|u|+\eta|u|=(a+\epsilon+\kappa\delta+\eta)|u|.
\end{eqnarray*}
As a result $u\in F(\beta,
[n(1-C_2(\Psi)\delta)],a+\epsilon+\kappa\delta+\eta,\Theta,\Upsilon).$
Thanks to our control of $|w^\prime|$, we can get
$$f(\beta, [n(1-C_2(\Psi)\delta)],a+\epsilon+\kappa\delta+\eta,\Theta,\Upsilon)
\geq
f(\alpha,n,\epsilon,\Phi,\Psi)m^{-\frac{c_1(n)+2nC_2(\Psi)\delta+C(\Upsilon)}{
 |\Upsilon_{\max}|}}.$$
This yields
 $$
 \Lambda_{\Phi}^{\Psi}(\alpha,\epsilon)\leq
 \frac{2C_2(\Psi)\log m}{|\Upsilon_{\max}|}\delta+(1-C_2(\Psi)\delta)
 \Lambda_{\Theta}^{\Upsilon}(\beta, a+\epsilon+\kappa\delta+\eta).
 $$
 Letting $\epsilon\downarrow 0$, then
$a\downarrow a_0$,  and $\delta\downarrow \delta_0$ we get
\begin{eqnarray*}
\Lambda_{\Phi}^{\Psi}(\alpha)&\leq& \frac{2C_2(\Psi)\log
m}{|\Upsilon_{\max}|}\delta_0+(1-C_2(\Psi)\delta_0)\Lambda_{\Theta}^{\Upsilon}(\beta,
(a_0+\kappa\delta_0+\eta)+)\\
&\leq& \frac{2C_2(\Psi)\log
m}{|\Upsilon_{\max}|}\delta_0+(1-C_2(\Psi)\delta_0)\Lambda_{\Theta}^{\Upsilon}(\beta,
a_0+\kappa\delta_0+2\eta).
\end{eqnarray*}

\section{Proof of Theorem~\ref{main-fun-level-one-sided}}\label{proofmainth}

We prove the slighly more general result mentioned in
Remark~\ref{remloc}(2). Suppose that $\xi$ is  continuous outside a
subset $E$ of $\Sigma_A$, bounded and $\xi(\Sigma_A)\subset
\mathrm{aff}(L_\Phi)$. Also, suppose that
 $\dim_H E<\lambda:=\sup\{\D(\alpha):\alpha\in \xi (\Sigma_A\setminus E)\cap
 \mathrm{ri}(L_\Phi)\}.$
 To prepare the proof of our geometric results,
  we need the following more general result.

 \begin{proposition}\label{lower-unify-func}
 Let $Z\subset \Sigma_A$ be a
closed set such that $\mu(Z)=0$ for any Gibbs measure $\mu$
 fully supported on $\Sigma_A$. For any $\delta>0$ such that
  $\lambda-\delta> \dim_H E$,
  we can construct a Moran subset $\Theta\subset\Sigma_A$
  such that $\Theta\setminus E\subset
 E_\Phi(\xi)$, $\dim_H \Theta\geq \lambda-\delta$ and  there
 exists an increasing sequence of integers $(\tilde g_j)_{j\ge 1}$
 such that $T^{\tilde g_j}x\not\in Z$ for any $x\in \Theta$ and any
 $j\geq 1.$
 \end{proposition}
\begin{proof} Fix $\delta>0$ such that $\lambda-\delta> \dim_H E$.
Choose $\alpha_0\in \xi(\Sigma_A\setminus E)\cap \text{ri}(L_\Phi)$
such that $\D(\alpha_0)>\lambda-\delta/2.$ Since $\D$ is continuous
in $\text{ri}(L_\Phi)$, we can find $\eta>0$ such that $\widetilde
B_\eta:=B(\alpha_0,\eta)\cap \text{aff}(L_\Phi)\subset
\text{ri}(L_\Phi)$ and for any $\alpha\in \widetilde B_\eta$ we have
$ |\D(\alpha)-\D(\alpha_0)|<\delta/2. $ Consequently
\begin{equation}\label{loc-0}
\D(\alpha)>\lambda-\delta\ \ \text{ for all }\ \ \alpha\in
\widetilde B_\eta.
\end{equation}

Now we proceed in four steps. The two first steps provide the scheme
of the construction of the set $\Theta$ and a good measure $\rho$, a
piece of which is supported by $\Theta$. The next two steps complete
the construction to ensure that $\Theta$ as the required properties
and the dimension of $\rho$ restricted to $\Theta$ has a Hausdorff
dimension larger than or equal to $\lambda-\delta$.

 {\bf Step 1:} {\it Concatenation of measures.} Assume $L_\Phi$
 has dimension $d_0\leq d$ and ${\rm
aff}(L_\Phi)=\alpha_0+U(\R^{d_0}\times \{0\}^{d-d_0})$, where $U$ is
a $d\times d$ orthogonal matrix.
 Let $j_0\in\N$ such that
$2^{-j_0}\sqrt{d_0}<\eta$ and define a sequence of sets as follows:
$$
\Delta_j:=\widetilde B_\eta\cap (\alpha_0+2^{-j-j_0}U
(\Z^{d_0}\times \{0\}^{d-d_0})),\ \ j\geq 0.
$$
Then $\Delta_0\ne\emptyset,$ $\Delta_j\subset\Delta_{j+1}$ for any
$j\geq 0$ and each $\Delta_j$ is a finite set. For each $\alpha\in
\bigcup_{j\geq 0}\Delta_j$, we can find a measure $\mu_\alpha$ such
that
\begin{equation}\label{loc-1}
{\Phi}_\ast(\mu_\alpha)=\alpha\ \ \text{ and }\ \ \ \
\D(\alpha)={\mathcal
E}(\alpha)=\frac{h_{\mu_\alpha}}{\gamma_\alpha}, \ \ \text{ where }\
\  \gamma_\alpha=-\Psi_\ast(\mu_\alpha).
\end{equation}

Let $(\varepsilon_j)_{j\ge 1}\in (0,1)^\N$  such that
$\sum_j\eps_j<\infty$. For each $j\ge 1$ define
$$
\overline \Delta_j=\{(\alpha,j):\alpha\in \Delta_j\}.
$$
For each $\sigma=(\alpha,j)\in \overline \Delta_j$, we can find a
Markov (hence Gibbs) measure $\mu_{\sigma} $ such that
\begin{equation}\label{loc-2}
\max (|h_{\mu_{\sigma}}-h_{\mu_\alpha}|,\ |\beta_{\sigma}-\alpha|,\
|\gamma_{\sigma}-\gamma_\alpha|)< \eps_j,
\end{equation} where
$\beta_{\sigma}={\Phi}_\ast(\mu_{\sigma})$ and
$\gamma_{\sigma}=-\Psi_\ast(\mu_{\sigma})$.

Let $(\phi^j)_{j\ge 1}$ and $(\psi^j)_{j\ge 1}$ be two sequences of
H\"older potentials defined on $\Sigma_A$ such that
\begin{equation}\label{phi-psi}
\|\Phi^{(j)}- \Phi\|_{\rm{lim}}< \eps_j\ \ \text{ and }\ \
\|\Psi^{(j)}- \Psi\|_{\rm{lim}}< \eps_j,
\end{equation}
where $\Phi^{(j)}=(S_n\phi^j)_{n=1}^{\infty}$ and
$\Psi^{(j)}=(S_n\psi^j)_{n=1}^{\infty}$.

For each $\omega=(\sigma,s)\in \overline \Delta_j\times
\{1,\cdots,m\},$ we denote by $\mu_\omega$ the restriction of
$\mu_{\sigma}$ to $[s]$ and
 $\nu_\omega$ the probability measure
 $\mu_\omega/\mu_{\omega}([s])$.

Now fix any positive integer sequence $\{L_j\}_{j\geq 1}$, we will
build a concatenated measure $\rho$ on $\Sigma_A$ with support
contained in a small cylinder $[\vartheta]$. At first we define
$\vartheta\in\Sigma_{A,*}$ and inductively a sequence of integers
$\{g_j:j\geq 0\}$ and a sequence of measures $\{\rho_j:j\geq 0\}$
such that $\rho_j$ is a measure on
$\Big([\vartheta],\sigma\{[u]:\vartheta\prec u\in
\Sigma_{A,g_j}\}\Big)$ for each $j\ge 0$, and the measures $\rho_j$
are consistent: for each $j\ge 0$ the restriction of $\rho_{j+1}$ to
$\sigma\{[u]:\vartheta\prec u\in \Sigma_{A,g_j}\}$ is
 equal to~$\rho_j$.

Fix $\hat x\in \Sigma_A\setminus E$ such that $\xi(\hat x)=\alpha_0$
and write $\hat x=\hat x_1\hat x_2\cdots.$ Since $\xi$ is continuous
at $\hat x$, we can choose $g_0\in\N$ such that
$\text{Osc}(\xi,[\hat x|_{g_0}])\leq 2^{-j_0}$, where
$\text{Osc}(\xi,V)$ stands for the oscillation of $\xi$ over $V$.
Write $\vartheta:=\hat x|_{g_0}.$ Define the probability measure
$\rho_0$ to be the trivial probability measure on
$([\vartheta],\{\emptyset,[\vartheta]\})$.
 Suppose we
have defined $(g_k,\rho_k)_{0\le k\le j}$ for $j\ge 0$ as desired.
To obtain $(g_{j+1},\rho_{j+1})$ from $(g_j,\rho_j)$,  define
$g_{j+1}:=g_j+L_{j+1}$.  For every $w\in \Sigma_{A,g_j}$ with
$\vartheta\prec w$, choose $x_w\in[w]$. Since $x_w\in [w]\subset
[\vartheta]$ we have
$$
 |\xi(x_w)-\alpha_0|=|\xi(x_w)-\xi(\hat x)|\leq
2^{-j_0}\leq \eta.
$$
Notice that by our assumption $\xi(\Sigma_A)\subset
\text{aff}(L_\Phi)$, thus $\xi(x_w)\in \widetilde B_\eta.$ Take
$\alpha_w\in \Delta_{j+1}$ such that $|\xi(x_w)-\alpha_w|\leq
2^{-j-1-j_0}\sqrt{d_0}$. Let $\omega=(\alpha_w,j+1,t_w)$, where for
each $u\in\Sigma_{A,\ast}$, $t_u$ stands for the last letter of $u$,
and it is called the {\it type} of $u.$   Then $\omega\in \overline
\Delta_{j+1}\times \{1,\cdots,m\}.$ For each
$v\in\Sigma_{A,L_{j+1}}$ such that $wv$ is admissible, define
$$
\rho_{j+1}([wv]):=\rho_j([w])\nu_\omega([t_{w}v]).
$$
By construction the family $\{\rho_j:j\geq 0\}$ is consistent.
Denote by $\rho$ the Kolmogorov extension of the sequence
$(\rho_j)_{j\ge 0}$ to $([\vartheta],\sigma\{[u]:\vartheta\prec u\in
\Sigma_{A,*}\})$. This finishes the construction of the desired
measure. Note that $g_j=g_0+L_1+\cdots+L_j$ for any $j\geq 1.$ Also,
by construction, we have the following formula for the $\rho$-mass
of any cylinder of generation larger than $g_0$. If $\vartheta\prec
u$ and $u\in \Sigma_{A,n}$ with $g_j\le n<g_{j+1}$, writing
$u=\vartheta w^1\cdots w^j\cdot v$ with $|w^k|=L_k$ and $|v|=n-g_j,$
and denoting $\vartheta w^1\cdots w^k$ by $\widetilde{w}_k$, then
\begin{equation}\label{rho-fun-level}
 \rho([u])=\Big( \prod_{k=1}^{j}
 \nu_{\omega_k}([t_{w^{k-1}}w^{k}])\Big )
 \nu_{\omega_{j+1}}([t_{w^{j}}v]),
\end{equation}
where $\omega_k=(\alpha_{\widetilde w_{k-1}},k,t_{w^{k-1}})$ for
$k=1,\cdots,j+1.$

{\bf Step 2:} {\it Construction of the Moran set $\Theta.$}

Next we want to specify the integer sequence $\{L_j\}_{j\geq 1}$ and
pick out carefully a Moran set $\Theta\subset [\vartheta]$ such that
$\rho(\Theta)>0$ and $\Theta$ has the last property stated in the
proposition. We proceed as follows.

 Fix $\omega=(\alpha,j,s)\in\overline \Delta_j\times
\{1,\cdots,m\}.$ For $N\ge 1$ Define
\begin{eqnarray*}
E_N(\omega)&:=&\bigcap_{n\ge N}\Big\{x\in [s]: \Big
|\frac{S_n\phi^j(Tx)}{n}-\alpha\Big |,\  \Big |\frac{\log
\nu_\omega([x|_{n}])}{-n}-h_{\mu_\alpha}\Big |,\\
&&\quad\quad\quad\quad\quad\quad\quad\quad\quad\quad\quad\quad\ \Big
|\frac{S_n\psi^j(Tx)}{-n}-\gamma_\alpha\Big |\le 2\eps_j\Big\}.
\end{eqnarray*}
 Notice
that each $\overline\Delta_j$ is a finite set, thus by   the
ergodicity of each $\mu_{(\alpha,j)}$, $(\ref{loc-2})$ and
$(\ref{phi-psi})$, we can fix an integer $N_j$ such that
\begin{equation}\label{aaa}
 \nu_\omega(E_N(\omega))\ge
1-\eps_j/2,\ \ \ \ (\forall \ N\ge N_j, \ \ \forall
\omega\in\overline \Delta_j\times \{1,\cdots,m\}).
\end{equation}
Define $ V_N:=\big \{v\in \Sigma_{A,N+1}:  [v]\cap Z=\emptyset\big
\}. $ By the restriction about $Z,$ there exists an integer
$\widehat{N}_j$ such that
\begin{equation}\label{bbb}
\nu_\omega\Big( \bigcup_{v\in V_N}[v]\Big )\ge 1-\eps_j/2, \ \ \ \
(\forall \ N\ge \widehat N_j, \ \ \forall \omega\in\overline
\Delta_j\times \{1,\cdots,m\}).
\end{equation}
Define $
 V_N(\omega)=\{v\in V_N,\ [v]\cap E_{N_j}(\omega)\neq\emptyset\}
.$ Thus, if $N\ge \max(N_j,\widehat{N}_j)$, by $(\ref{aaa})$ and
$(\ref{bbb})$ we have
\begin{equation}\label{ccc}
\nu_\omega\Big (\bigcup_{v\in  V_N(\omega)}[v]\Big )\ge 1-\eps_j,\ \
\ \ (\forall \omega\in\overline \Delta_j\times \{1,\cdots,m\}).
\end{equation}

Take  an integer sequence  $\{L_j\}_{n\geq 1} $  such that $L_j\geq
\max(N_j,\widehat{N}_j)$ for $j\geq 1$ and consider the associated
measure $\rho$ constructed in step 1.  We define the desired Moran
set as
$$
\Theta:=\big \{x\in [\vartheta]: \ \forall\ j\ge 0, \
T^{g_{j}-1}x|_{L_{j+1}+1}\in
V_{L_{j+1}}(\alpha_{x|_{g_{j}}},j+1,x_{g_{j}})\ \big\}.
$$ By construction, $T^{g_{j}-1}x\not\in Z$ for any $x\in
\Theta$ and any $j\geq 0.$ Define $\tilde g_j=g_j-1,$ we checked the
last property of $\Theta.$

Write $\omega_k:=(\alpha_{x|_{g_{k-1}}},k,x_{g_{k-1}})\in\overline
\Delta_k\times\{1,\cdots,m\}$. For each $j\ge 1$, by using
\eqref{rho-fun-level} and \eqref{ccc} we  get
\begin{eqnarray*}
&&\rho(\{x\in [\vartheta]: [x|_{g_j}]\cap \Theta\neq\emptyset\})
=\sum_{ \substack{ \vartheta w^1\cdots w^j\ \text{admissible},\\
\forall\, 1\le k \le j,\,  x_{g_{k-1}}w^k\in
 V_{L_k}(\omega_k)}} \rho_j(\vartheta w^1\cdots w^j)\\
&=&\sum_{  \substack{ \vartheta w^1\cdots w^{j-1}\ \text{admissible},\\
\forall\, 1\le k \le j-1,\,  x_{g_{k-1}}w^k\in
V_{L_k}(\omega_k)}}\rho_{j-1}(\vartheta w^1\cdots w^{j-1})
\sum_{x_{g_{j-1}}w^j\in   V_{L_j}(\omega_j)}
\nu_{\omega_j}([x_{g_{j-1}} w^j])\\
&\ge &\sum_{  \substack{ \vartheta w^1\cdots w^{j-1}\ \text{admissible},\\
\forall\, 1\le k \le j-1,\,  x_{g_{k-1}}w^k\in V_{L_k}(\omega_k)}}
\rho_{j-1}(\vartheta w^1\cdots w^{j-1})(1-\eps_j)\ge
\prod_{k=1}^j(1-\eps_k).
\end{eqnarray*}
  Since we
assumed that $\eps_j<1$ and $\sum_{j\ge 1}\eps_j<\infty$ we have
$$
\rho(\Theta)=\lim_{j\to\infty}\rho(\{x\in [\vartheta]:
[x|_{g_j}]\cap \Theta\neq\emptyset\})\ge \prod_{j\ge 1}(1-\eps_j)>0.
$$

{\bf Step 3:} {\it Conditions on $(L_j)_{j\geq 1}$ ensuring that
$\Theta\setminus E\subset E_\Phi(\xi).$}

 For
$\eta\in \{\phi,\psi\}$ and $j\ge 1$, let
$$
{\rm Var}(\eta^j)=\sup_{n\ge 1}\ \max_{v\in\Sigma_{A,n}}\
\max_{x,y\in[v]}\ |S_n\eta^j(x)-S_n\eta^j(y)|.
$$
This number is finite since each $\eta^j$ is H\"older continuous. By
$(\ref{phi-psi})$ we can find integer sequence $M_j\nearrow \infty$
such that
\begin{equation}\label{m-j}
\max(\|S_n\phi^j-\phi_n\|_\infty,\|S_n\psi^j-\psi_n\|_\infty)\le
2\eps_j n\ \ \ \ \ \  (\forall\ n\ge M_j).
\end{equation}
 The sequence $(L_j)_{j\ge 1}$ can be specified to satisfy the additional properties
$$
 L_j\ge M_{j+1} \ \ \text{and }\ \
\max(K_1(j), K_2(j),K_3(j))\le \eps_jg_j,
$$
(recall that $g_j=g_0+\sum_{k=1}^j L_k$), where
$$
\begin{cases}
K_1(j)=&\sum_{k=1}^{j+1}({\rm Var}(\phi^k)+{\rm Var}(\psi^k));\\
K_2(j)=&\max_{\substack{\alpha\in\Delta_{j+1}\\1\le s\le m\\1\le
n\le N_{j+1}}}\max\Big(n|\alpha|,\|S_n\phi^{j+1}\|_\infty,\|
\log\nu_{(\alpha,j+1,s)}([\cdot|_{n}])\|_\infty,\\
&\quad\quad\quad\quad\quad\quad\quad\quad\quad\|S_n\psi^{j+1}\|_\infty\Big);\\
K_3(j)=&(j+1)\max_{1\le n\le M_{j+1}} \max (\|S_n\phi^{j+1}-
\phi_{n}\|_\infty,\|S_n\psi^{j+1}-\psi_n\|_\infty).
\end{cases}
$$

 Let us check that $\Theta\setminus E\subset
E_{{\Phi}}({\xi}).$ Let $x\in\Theta\setminus E$, $n\ge g_1$ and
$j\ge 1$ such that $g_j\le n<g_{j+1}$. Since $n\geq g_j>L_j\ge
M_{j+1}$, by $(\ref{m-j})$ we have
\begin{eqnarray*}
|\phi_n (x)-n\xi(x)|&\le& \|S_n\phi^{j+1}-\phi_n\|_\infty+
|S_n\phi^{j+1}(x)-n\xi(x)| \\
&\le& 2\eps_{j+1} n+|S_n\phi^{j+1}(x)-n\xi(x)|.
\end{eqnarray*}
For the second term in the right hand side we have (with $g_{-1}=0$,
$\alpha_k=\alpha_{x|_{g_{k-1}}}$ and $L_0=g_0$)
\begin{eqnarray*}
&&|S_n\phi^{j+1}(x)-n\xi(x)|\\
&\leq&|S_{g_j}\phi^{j+1}(x)-g_j\xi(x)|+|S_{n-g_j}\phi^{j+1}(T^{g_j}x)-(n-g_j)\xi(x)|\\
&=&|S_{g_j}\phi^{j+1}(x)-\sum_{k=0}^jL_k\alpha_k+\sum_{k=0}^jL_k\alpha_k-
  g_j\xi(x)|\\
  &&+|S_{n-g_j}\phi^{j+1}(T^{g_j}x)-(n-g_j)\xi(x)|\\
  &\leq&\sum_{k=0}^j|S_{L_k}\phi^{j+1}(T^{g_{k-1}}x)-L_k\alpha_k|+
  \sum_{k=0}^jL_k|\alpha_k-
  \xi(x)|\ \ \ \big(=: (I)+(II) \big)\\
  &&+|S_{n-g_j}\phi^{j+1}(T^{g_j}x)-(n-g_j)\alpha_{j+1}|\quad\quad
  \quad\quad\quad\quad\quad\quad\quad \ \big(=: (III) \big)\\
  &&+(n-g_j)|\alpha_{j+1}-\xi(x)|
  \quad\quad\quad\quad\quad\quad\quad\quad
  \quad\quad\quad\quad\quad\quad\quad \ \big(=: (IV) \big).
\end{eqnarray*}

At first we have
\begin{eqnarray*}
&&(I)+(III)\\ &\le&
\sum_{k=0}^j\|S_{L_k}\phi^{j+1}-\phi_{L_k}\|_\infty+
\sum_{k=0}^j\|S_{L_k}\phi^k-\phi_{L_k}\|_\infty\\
&&+\Big ( \sum_{k=0}^j|S_{L_k}\phi^k(T^{g_{k-1}}x)- L_k\alpha_k|\Big
)+ |S_{n-g_j}\phi^{j+1}(T^{g_j}x)-(n-g_j)\alpha_{j+1}|.
\end{eqnarray*}
If $L_k\leq M_{j+1}$, then
$$
\|S_{L_k}\phi^{j+1}-\phi_{L_k}\|_\infty\leq K_3(j)/(j+1);
$$
 if $L_k>
M_{j+1}$, then
$$
\|S_{L_k}\phi^{j+1}-\phi_{L_k}\|_\infty\leq 2\varepsilon_kL_k.
$$ Thus we have
\begin{equation}\label{theta-1}
\sum_{k=0}^j\|S_{L_k}\phi^{j+1}-\phi_{L_k}\|_\infty\le
K_3(j)+2\sum_{k=0}^j\eps_kL_k\leq
\epsilon_jg_j+2\sum_{k=0}^j\eps_kL_k.
\end{equation}
Since $L_k\geq M_{k+1}\geq M_k$ we also have
\begin{equation}\label{theta-2}
\sum_{k=0}^j\|S_{L_k}\phi^k-\phi_{L_k}\|_\infty\le
2\sum_{k=0}^j\eps_kL_k.
\end{equation}
 Thus both terms are $o(g_j)$ as
$n\to\infty$. Consequently both terms are $o(n).$

Write $\omega_k=(\alpha_k,k,x_{g_{k-1}})$ for $k=0,\cdots,j$. By the
construction of $\Theta$, we have
$$
T^{g_{k-1}-1}x|_{L_k+1}=x_{g_{k-1}}\cdot (T^{g_{k-1}}x|_{L_k})\in
V_{L_k}(\omega_k),
$$ so
 $[x_{g_{k-1}}\cdot (T^{g_{k-1}}x|_{L_k})]\cap E_{N_k}(\omega_k)\neq\emptyset$.
 Using the definition of $E_{N_k}(\omega_k)$, since $L_k\geq
 N_k,$
there exists $y\in [T^{g_{k-1}}x|_{L_k}]$ such that
$|S_{L_k}\phi^k(y)-L_k\alpha_k|\le 2\eps_kL_k$, hence
 $$
 |S_{L_k}\phi^k(T^{g_{k-1}}x)-L_k\alpha_k|\le
 2\eps_kL_k+{\rm Var}(\phi^k).
 $$
 Similarly if $ n-g(j)\ge N_{j+1}$,
$$
|S_{n-g_j}\phi^{j+1}(T^{g_j}x)-(n-g_j)\alpha_{j+1}|\le
2\eps_{j+1}(n-g_j)+{\rm Var}(\phi^{j+1}).
$$
If $ n-g(j)\le N_{j+1}$ we trivially have
$$
|S_{n-g_j}\phi^{j+1}(T^{g_j}x)-(n-g_j)\alpha_{j+1}|\le 2K_2(j).
$$
This yields
\begin{eqnarray}
\nonumber&& \Big ( \sum_{k=0}^j|S_{L_k}\phi^k(T^{g_{k-1}}x)-
L_k\alpha_k|\Big )+
|S_{n-g_j}\phi^{j+1}(T^{g_j}x)-(n-g_j)\alpha_{j+1}|\\
\nonumber&\le &
 2\sum_{k=0}^j\eps_kL_k+\sum_{k=0}^{j+1}{\rm Var}(\phi^k)+
  2\eps_{j+1}(n-g_j)+2K_2(j)\\
 \label{theta-3}&\le&2\sum_{k=0}^j\eps_kL_k+ K_1(j)+  2\eps_{j+1}(n-g_j)+2K_2(j)
  =o(g_j)+o(n)=o(n).
 \end{eqnarray}
Combining $(\ref{theta-1}), (\ref{theta-2})$ and $(\ref{theta-3})$
we get $(I)+(III)=o(n)$.

On the other hand, by construction (recall that
$\alpha_{k+1}=\alpha_{x|_{g_k}}$)
$$
|\xi(x)-\alpha_{k+1}|\leq
|\xi(x)-\xi(x_{x|_{g_k}})|+|\xi(x_{x|_{g_k}})-\alpha_{x|_{g_k}}|\leq
{\rm Osc}(\xi,[x|_{g_k}])+2^{-k-j_0}\sqrt{d_0},
$$
where $x_{x|_{g_k}}$ is the special point in $[x|_{g_k}]$ chosen in
the construction of the measure $\rho$. Since $x\not\in E,$ $\xi$ is
continuous at $x$. We have
$$
\lim_{k\to\infty}{\rm Osc}(\xi,[x|_{g_k}])+2^{-k-j_0}\sqrt{d_0}=0.
$$ Thus $\lim_{k\to\infty}\alpha_k=\xi(x)$ and
we conclude that $(II)+(IV)=o(n).$ As a result $| \phi_n
(x)-n\xi(x)|=o(n)$, thus $x\in E_\Phi(\xi)$. This finishes the proof
of $\Theta\setminus E\subset E_{ \Phi}({\xi})$.

{\bf Step 4:}  $\dim_H\Theta\geq \lambda-\delta.$

Let us compute the local lower dimension $\underline{d}_\rho(x)$ for
any $x\in \Theta.$ By using similar estimates as above, for any
$x\in\Theta$ we can prove that (with
$\sigma_k=(\alpha_{x|_{g_{k-1}}},k)$)
\begin{eqnarray*}
&&|-\psi_{n}(x)
-\sum_{k=0}^{j}L_k\gamma_{\sigma_k}-(n-g_j)\gamma_{\sigma_{j+1}}|=o(n),\\
&& |-\log\rho([x|_n])
-\sum_{k=0}^{j}L_kh_{\mu_{\sigma_k}}-(n-g_j)h_{\mu_{\sigma_{j+1}}}|=o(n).
\end{eqnarray*}
By $(\ref{loc-0}), (\ref{loc-1})$ and $(\ref{loc-2})$ we have
$\liminf_{j\to\infty} h_{\mu_{\sigma_j}}/\gamma_{\sigma_j}\geq
\lambda-\delta$. For any $y\in [x|_n]$ we have
$|\psi_n(y)-\psi_n(x)|=o(n)$, thus we get
$\text{diam}([x|_n])=\Psi[x|_n]=\exp(\psi_n(x)+o(n))$. Combining the
above two relations  we conclude that
$$
\underline{d}_\rho(x)=\liminf_{n\to\infty}\frac{\log\rho([x|_{n}])}{\log({\rm
diam}([x|_{n}])}\geq \lambda-\delta.
$$
 Since $\rho(\Theta)>0,$ by the mass distribution principle we get   $\dim_H\Theta\geq
\lambda-\delta$ (see \cite{P}).
\end{proof}

A slight modification of the above proof with $\xi$ taken constant
yields the following proposition:

 \begin{proposition}\label{lower-unify}
  Assume $Z\subset \Sigma_A$ is a
closed set such that $\mu(Z)=0$ for any Gibbs measure $\mu$ fully
supported on $\Sigma_A$. For any $\alpha\in L_\Phi,$ we can
construct a
 subset $\Theta\subset
 E_\Phi(\alpha)$ such that $\dim_H\Theta\ge {\mathcal E}(\alpha) $
 and  there exists an integer sequence $g_j\nearrow
 \infty$ such that $T^{g_j}x\not\in Z$ for any $x\in \Theta$ and any
 $j\geq 1.$ In particular, ${\mathcal E}(\alpha)\leq
 \D(\alpha)$.
 \end{proposition}

 \noindent{\bf Proof of Theorem~\ref{main-fun-level-one-sided}.}
(1')\ Since $\dim_H E<\lambda-\delta,$  by Proposition
\ref{lower-unify-func}  we have $\dim_H(\Theta\setminus E)=\dim_H
\Theta\geq\lambda-\delta.$ Consequently $ \dim_H E_\Phi(\xi)\geq
\dim_H (\Theta\setminus E)\geq \lambda-\delta. $ Since $\delta>0$ is
arbitrary, we get $\dim_HE_\Phi(\xi)\geq \lambda$.

Now we assume $\xi$ is continuous everywhere.

(2) If $\xi(\Sigma_A)\subset L_\Phi$, the construction of a Moran
subset of
 $E_\Phi(\xi)$ can be done around any point of $\Sigma_A$, like in the proof of
 Proposition~\ref{lower-unify-func}. The only difference is that in this case
   the dimension of this set is of no importance. Hence, $E_\Phi(\xi)$  is dense.

(3)\  If
$$
\sup\{\D(\alpha):\alpha\in \xi (\Sigma_A)\cap \mathrm{ri}(L_\Phi)\}=
\sup\{\D(\alpha):\alpha\in
 {\xi (\Sigma_A)}\cap L_\Phi\}=:\theta,
$$ then at first we have $\dim_H E_\Phi(\xi)\geq \theta.$
On the other hand by  definition  we have $E_\Phi(\xi)\subset
E_\Phi( {\xi (\Sigma_A)}\cap L_\Phi)$. Thus by Proposition
\ref{upper-bound}, we have $\dim_P E_\Phi(\xi)\leq \theta.$ So we
get
$$
\dim_H E_\Phi(\xi)=\dim_P E_\Phi(\xi)= \theta.
$$

(4)\ Assume $d=1$ and $\xi(\Sigma_A)\subset L_\Phi.$ Notice that in
this case $L_\Phi=[\alpha_1,\alpha_2]$ is an interval. Thus by
Proposition \ref{regularity} and Theorem \ref{main-one-sided}, $\D$
is continuous on $L_\Phi.$ Assume $\alpha_0\in\xi(\Sigma_A)$ such
that $\D(\alpha_0)=\sup\{\D(\alpha):\alpha\in \xi(\Sigma_A)\}$. If
$\alpha_0\in (\alpha_1,\alpha_2),$ by (3) we conclude. Now assume
$\alpha_0=\alpha_1$. If $\alpha_1$ is not isolated in
$\xi(\Sigma_A),$ still by (3) and the continuity of $\D$, we get the
result. If $\alpha_1$ is isolated in $\xi(\Sigma_A)$, then by the
continuity of $\xi$, we can find a cylinder $[w]\subset \Sigma_A$
such that $\xi([w])=\alpha_1.$ From this we get $E_\Phi(\xi)\supset
E_\Phi(\alpha_1)\cap[w]$. Thus $\dim_H E_\Phi(\xi)\geq \D(\alpha_1)$
and the result holds. If $\alpha_0=\alpha_2$, the proof is the same.
 \hfill
$\Box$


\section{Proofs of results in section \ref{examples}}\label{sec7}

We will use the following lemma, which is standard and essentially
the same as Lemma 5.1 in  \cite{GP97} (the proof is elementary).
\begin{lemma} \label{dim-compare-lem}
Let $X$ and $Y$ be metric spaces and $\chi: X\to Y$ a surjective
mapping with the following property: there exists a function
$N:(0,\infty)\to \N$ with  $\log N(r)/\log r\to 0$ when $r\to 0$
such that for any $r>0$, the pre-image $\chi^{-1}(B)$ of any
$r$-ball in $Y$ can be covered by at most $N(r)$ sets in $X$ of
diameter less than $r$. Then for any set $E\subset Y$ we have
$\dim_H E\geq \dim_H \chi^{-1}(E).$
\end{lemma}

\noindent{\bf Proof of Proposition \ref{dim-compare-prop}.} \
Condition (4) implies that $\chi:(\Sigma_A,d_\Psi)\to (J,d)$ is
Lipschitz continuous, thus we have $\dim_H E\leq
\dim_H\chi^{-1}(E).$

For the converse inequality, let us check the condition of the above
lemma. Let $B\subset J$ be a ball of radius $r$, let $n\in \N$ such
that $e^{-n}\leq r< e^{1-n}$. Define
 $$
G_B^r=\{w\in {\mathcal B}_n(\Psi): R_w\cap B\ne \emptyset\}.
 $$
 One checks that $\{[w]:w\in G_B^r\}$ is an $r$-covering of
 $\chi^{-1}(B)$. Define $N(r):=\#G_B^r.$
  Let us estimate the number $\#G_B^r.$ Clearly, $\#G_B^r\geq 1.$
  By  condition (4), for each $w\in
 G_B^r$, $R_w$ is contained in a ball of radius $K\Psi[w]\leq
 Ke^{-n}$, thus
  $$
  \bigcup_{w\in G_B^r}R_w\subset B(y,r+2Ke^{-n})\subset
  B(y,(e+2K)e^{-n}),
  $$
  where $y$ is the center of $B$. On
 the other hand, by Lemma \ref{Moran-covering}(1) there exists $C>0$
 such that $|w|\leq
 Cn$ for any $w\in{\mathcal B}_n(\Psi)$, thus
 $\eta_{|w|}=o(|w|)=o(n)$ for any $w\in{\mathcal B}_n(\Psi)$.
By  construction, the interiors of the sets $R_w$, $ w\in G_B^r$,
are disjoint and each $R_w$ contains a ball
 of radius
 $$
 K^{-1}\exp(\eta_{|w|})\Psi[w]=K^{-1}e^{o(n)}\Psi[w]=K^{-1}e^{-n+o(n)}
 $$
 by Lemma
 \ref{Moran-covering} (2). Thus $\#G_B^r\leq K^{d^\prime}(e+2K)^{d^\prime}e^{o(n)}$.
 So we conclude that $\log N(r)/\log r=\log \#G_B^r/\log r\to 0$
 as $r\to 0.$ Thus by lemma \ref{dim-compare-lem}, we  can conclude
 that $\dim_H E\geq \dim_H\chi^{-1}(E).$
\hfill$\Box$

\noindent{\bf Proof of Lemma \ref{boundary}.}\ At first we show that
$J\cap V= J\setminus \widetilde Z_\infty,$ consequently by the SOSC,
$J\setminus \widetilde Z_\infty\ne\emptyset$ and  we get
$\emptyset\ne \chi^{-1}(J\setminus \widetilde
Z_\infty)=\Sigma_A\setminus Z_\infty.$ In fact
\begin{eqnarray*}
y\in J\setminus \widetilde Z_\infty &\Leftrightarrow& y\in J \text{
and }  \forall\ n\ge 1\ \exists\ x\in\Sigma_A \ \ s.t.\ \ y\in
\text{int} (R_{x|_n})=f_{x|_n}(V) \\
&\Leftrightarrow& y\in J  \text{ and } \forall\ n\ge 1\ \exists!\
x\in\Sigma_A \ \ s.t.\ \ y\in \text{int}
(R_{x|_n})=f_{x|_n}(V)\\
&\Leftrightarrow& y\in J\cap V.
\end{eqnarray*}

By construction, $\chi: \Sigma_{A}\setminus  Z_\infty\to J\setminus
 \widetilde Z_\infty$ is surjective. Since $ J\setminus \widetilde
Z_\infty=J\cap V,$ it is ready to show that $\chi$ is also
injective.

Next we show that $T(\Sigma_A\setminus Z_\infty)\subset
\Sigma_A\setminus Z_\infty.$ Take $x\in \Sigma_A\setminus Z_\infty$.
If
 $Tx\in Z_\infty,$ then we can find $n_0\in \N$ such that
$\chi(Tx)\in f_{Tx|_{n_0}}(\partial V)$. Consequently $
\chi(x)=f_{x_1}(\chi(Tx))\in f_{x_1}(f_{Tx|_{n_0}}(\partial V))=
f_{x_1}\circ f_{Tx|_{n_0}}(\partial V)= f_{x|_{n_0+1}}(\partial V)
$, which is a contradiction. Next we show that for any Gibbs measure
$\mu$ we have $\mu(Z_\infty)=0.$ Define $ \widetilde
Z_n:=\bigcup_{w\in \Sigma_{A,\ast},\ |w|\leq n}f_w(\partial V) $ and
$Z_n=\chi^{-1}(\widetilde Z_n)$. The sequence $(Z_n)_{n\ge 1}$ is
non decreasing and $Z_\infty=\bigcup_{n\geq 1}Z_n.$ Since the IFS is
conformal we can easily get $T(Z_n)\subset Z_{n-1}$ for $n\geq 1$
and $T(Z_0)\subset Z_0.$ Consequently $T(Z_n)\subset Z_{n}.$ By the
ergodicity we have $\mu(Z_n)=0$ or $1$. By the SOSC,
  $\Sigma_A\setminus Z_n$ is nonempty and open, thus  by
the Gibbs property of $\mu$ we get $\mu(\Sigma_A\setminus Z_n)>0$,
hence $\mu(Z_n)=0$. Consequently $\mu(Z_\infty)=0.$ At last, from
$T(Z_n)\subset Z_{n}$ we easily get $T(Z_\infty)\subset Z_\infty.$
\hfill$\Box$

\noindent{\bf Proof of Theorem~\ref{appli-one-sided}.}\ (1) At first
we notice that by the property (4) assumed in the construction of
$J$ the mapping $\chi$ is Lipschitz. This is enough to get the
desired upper bounds from Theorem~\ref{main-one-sided}(1).

Now we deal with the lower bound for dimensions and the equality
$L_\Phi=L_{\widetilde\Phi}$.
 We notice that the inclusion $L_\Phi\subset L_{\widetilde\Phi}$ holds by construction.

\noindent{\it Suppose $J$ is a conformal repeller.} Since we have
$\chi\circ T=g\circ \chi$ on $\Sigma_A$ and $\chi$ is surjective, it
is seen that
$\chi^{-1}(E_\Phi(\alpha))=E_{\widetilde{\Phi}}(\alpha)$ for any
$\alpha\in L_{\widetilde{\Phi}}.$ Thus $L_\Phi=L_{\widetilde{\Phi}}$
and by Proposition \ref{dim-compare-prop}, we have $\dim_H
E_\Phi(\alpha)=\dim_H E_{\widetilde \Phi}(\alpha).$

\noindent{\it Suppose $J$ is the attractor of a  conformal IFS with
SOSC.} Let $\alpha\in L_{\widetilde \Phi}$. Let
$Z=\chi^{-1}(\partial V)$. The set $Z$ is closed and by Lemma
\ref{boundary},
  $\mu(Z)=0$ for any Gibbs
measure $\mu$. By Proposition \ref{lower-unify} we can construct a
Moran set $\Theta\subset E_{\widetilde{\Phi}}(\alpha)$ such that
$\dim_H(\Theta)\geq {\mathcal E}_{\widetilde \Phi}(\alpha)=\dim_H
E_{\widetilde \Phi}(\alpha)$ and there exists a sequence
$g_j\nearrow\infty$ such that $T^{g_j}x\not\in Z$ for any
$x\in\Theta$ and any $j\geq 1.$ The last property means that
$\Theta\subset \Sigma_A\setminus Z_\infty.$ Since $\chi$ is a
bijection between $\Sigma_A\setminus Z_\infty$ and
$J\setminus\widetilde Z_\infty$, we conclude that
$\chi^{-1}\circ\chi(\Theta)=\Theta,$ thus by Proposition
\ref{dim-compare-prop}, $ \dim_H\chi(\Theta)=\dim_H \Theta\geq\dim_H
E_{\widetilde \Phi}(\alpha).$ Since we also have $\chi\circ
T=\widetilde{g}\circ\chi$ on $\Sigma_A\setminus Z_\infty,$ we get
that $\chi(\Theta)\subset E_\Phi(\alpha).$ Thus $\alpha\in L_\Phi$
and $\dim_H E_\Phi(\alpha)\geq \dim_H E_{\widetilde \Phi}(\alpha)$.

\noindent (2) Take $E=J$ in Proposition \ref{dim-compare-prop}, then
use (1) and Theorem \ref{main-one-sided}(2). \hfill $\Box$

\noindent{\bf Proof of Theorem~\ref{appli-fun-level-one-sided}.}\
Define $\widetilde{\xi}:=\xi\circ\chi.$

\noindent{\it Case 1:}  $J$ is a conformal repeller.  One checks
easily that
$\chi^{-1}(E_\Phi(\xi))=E_{\widetilde{\Phi}}(\widetilde{\xi}).$ Then
the result is a consequence of Theorem \ref{appli-one-sided},
Theorem \ref{main-fun-level-one-sided} and Proposition
\ref{dim-compare-prop}.

\noindent{\it Case 2:}  $J$ is the attractor of a conformal IFS with
SOSC. We conclude by using Proposition \ref{lower-unify-func},
Theorem \ref{main-fun-level-one-sided}, Proposition
\ref{dim-compare-prop} and the same argument as in the proof of
Theorem \ref{appli-one-sided}.
 \hfill$\Box$

\noindent{\bf Proof of Theorem \ref{carpet}.}\ Let $\xi=\Phi,$ then
${\mathcal
 F}(\widetilde{J},\widetilde{g})=E_{\Phi}(\xi)$. To show the result we need
 only to check the condition of Theorem
 \ref{appli-fun-level-one-sided} and the only condition we need to
 check is that
\begin{equation}\label{final}
\sup\{\D_\Phi(\alpha): \alpha\in \xi(J)\cap{\rm
ri}(L_\Phi)\}=\sup\{\D_\Phi(\alpha): \alpha\in \xi(J)\cap L_\Phi\}.
\end{equation}
Notice that in this special case we have $\xi(J)=J$ and $L_\Phi={\rm
Co}(J),$ thus $\xi(J)\cap L_\Phi=J.$  Recall  that in this case
$L_\Phi$ is a convex polyhedron, thus by Proposition
\ref{regularity} and Theorem \ref{appli-one-sided}, $\D_\Phi$ is
continuous on $L_\Phi.$ Thus the supremum  in the right hand side of
$(\ast)$ can be reached. If the maximum is attained in  ${\rm
ri}(L_\Phi)$, then the result is obvious. Now suppose that there
exists  $\alpha_0\in
 \big(L_\Phi\setminus {\rm ri}(L_\Phi)\big)\cap J$
 such that $\D_\Phi(\alpha_0)=\sup\{\D_\Phi(\alpha):
\alpha\in J\}.$
 By the structure of $J$, it is ready to see that $B(\alpha_0,\delta)\cap
J\cap {\rm ri}(L_\Phi)\ne \emptyset$ for any $\delta>0$. By the
continuity of $\D_\Phi$, $(\ref{final})$ holds immediately.
 \hfill $\Box$


\medskip
\medskip

\end{document}